\documentclass[a4paper, 11pt]{amsart}
\usepackage{geometry}
\usepackage[latin1]{inputenc}
\usepackage[T1]{fontenc}
\usepackage{amssymb}
\usepackage{amsmath}
\usepackage{amsthm}
\usepackage{mathrsfs}
\usepackage{amssymb}
\usepackage{tikz-cd}
\usepackage[hidelinks]{hyperref}
\usepackage[nameinlink, capitalise]{cleveref}

\newtheorem{LetterTheorem}{Theorem}

\newtheorem{Theorem}{Theorem}[section]
\newtheorem{Proposition}[Theorem]{Proposition}
\newtheorem{Lemma}[Theorem]{Lemma}
\newtheorem{Corollary}[Theorem]{Corollary}
\theoremstyle{definition}
\newtheorem{Remark}[Theorem]{Remark}
\newtheorem{Definition}[Theorem]{Definition}

\tikzset{close/.style={near start,outer sep=-2pt}}

\newcommand{\rmod}{\operatorname{mod}}
\newcommand{\modleft}[1]{#1\text{-mod}}
\newcommand{\projleft}[1]{#1\text{-proj}}
\newcommand{\Hp}[1]{\mathcal{H}(\mathrm{proj}\, #1 )}
\newcommand{\Hpleft}[1]{\mathcal{H}(#1\text{-proj})}
\newcommand{\HP}[1]{\mathcal{H}_{\mathcal{P}}\mathrm{(proj}\, #1 )}
\newcommand{\Hstp}[1]{\mathcal{H}_\mathrm{stp}\mathrm{(proj}\, #1 )}
\newcommand{\stmod}{\operatorname{\underline{mod}}}
\newcommand{\proj}{\operatorname{proj}}
\newcommand{\inj}{\operatorname{inj}}
\newcommand{\stp}{\operatorname{stp}}
\newcommand{\CC}[1]{\mathcal{C}(\rmod #1 )}
\newcommand{\CCp}[1]{\mathcal{C}(\proj #1 )}
\newcommand{\CCi}[1]{\mathcal{C}(\inj #1 )}
\newcommand{\CCleft}[1]{\mathcal{C}(#1\text{-mod})}
\newcommand{\CCpleft}[1]{\mathcal{C}(#1\text{-proj})}
\newcommand{\DDb}[1]{\mathcal{D}^b(\rmod #1)}
\newcommand{\DDpl}[1]{\mathcal{D}^+(\rmod #1 )}
\newcommand{\DDm}[1]{\mathcal{D}^-(\rmod #1 )}
\newcommand{\KK}[1]{\mathcal{K}(\rmod #1 )}
\newcommand{\KKleft}[1]{\mathcal{K}(#1\text{-mod})}
\newcommand{\KKp}[1]{\mathcal{K}(\proj #1 )}
\newcommand{\KKpleft}[1]{\mathcal{K}(#1\text{-proj})}
\newcommand{\KKb}[1]{\mathcal{K}^b(\rmod #1 )}
\newcommand{\KKbp}[1]{\mathcal{K}^b(\proj #1 )}
\newcommand{\KKbi}[1]{\mathcal{K}^b(\inj #1 )}
\newcommand{\KKbileft}[1]{\mathcal{K}^b(#1\text{-inj})}
\newcommand{\KKm}[1]{\mathcal{K}^-(\rmod #1 )}
\newcommand{\KKpl}[1]{\mathcal{K}^+(\rmod #1 )}
\newcommand{\KKpp}[1]{\mathcal{K}^+(\proj #1 )}
\newcommand{\KKi}[1]{\mathcal{K}(\inj #1 )}
\newcommand{\KKpi}[1]{\mathcal{K}^+(\inj #1 )}
\newcommand{\KKpbsp}[1]{\mathcal{K}^{+,b^\ast}\!(\proj #1 )}
\newcommand{\KKpbspleft}[1]{\mathcal{K}^{+,b^\ast}\!(#1\text{-proj})}
\newcommand{\KKpbi}[1]{\mathcal{K}^{+,b}(\inj #1 )}
\newcommand{\KKmbp}[1]{\mathcal{K}^{-,b}(\proj #1 )}
\newcommand{\KKptac}[1]{\mathcal{K}_{\mathrm{tac}}(\proj #1)}
\newcommand{\KP}[1]{\mathcal{K}^b(\PP_{#1})}
\newcommand{\Kstp}[1]{\mathcal{K}^b(\stp #1)}
\newcommand{\KPperp}[1]{{}^\perp \mathcal{K}^b(\mathcal{\PP}_{#1})}
\newcommand{\Kstperp}[1]{{}^\perp \mathcal{K}^b(\stp #1)}
\newcommand{\Pperp}{\!{}^\perp \mathcal{P}\!}
\newcommand{\stperp}[1]{{}^\perp (\stp #1)}

\newcommand{\HH}{\mathrm{H}}
\newcommand{\LL}{\mathcal{L}}
\newcommand{\TT}{\mathcal{T}}
\newcommand{\FF}{\mathcal{F}}
\newcommand{\GG}{\mathcal{G}}
\newcommand{\ZZ}{\mathbb{Z}}
\newcommand{\PP}{\mathcal{P}\!}

\newcommand{\tleq}{\tau_{\leqslant 0}\,}
\newcommand{\id}{\mathrm{id}}
\newcommand{\Du}{\operatorname{D}}
\newcommand{\Tr}{\operatorname{Tr}}
\newcommand{\im}{\operatorname{Im}}
\newcommand{\Ker}{\operatorname{Ker}}
\newcommand{\Cok}{\operatorname{Cok}}
\newcommand{\gdim}{\operatorname{gldim}}
\newcommand{\ddim}{\operatorname{domdim}}
\newcommand{\nddim}{\nu\operatorname{-domdim}}
\newcommand{\Hom}{\operatorname{Hom}}

\newcommand{\rad}{\operatorname{rad}}
\newcommand{\soc}{\operatorname{soc}}

\makeatletter
\newcommand*{\bt}{}
\DeclareRobustCommand*{\bt}{%
  {\mathbin{\mathpalette\bt@{}}}%
}
\newcommand*{\bt@scalefactor}{.75}
\newcommand*{\bt@widthfactor}{1.4}
\newcommand*{\bt@}[2]{%
  \sbox0{$#1\vcenter{}$}
  \sbox2{$#1\cdot\m@th$}%
  \hbox to \bt@widthfactor\wd2{%
    \hfil
    \raise\ht0\hbox{%
      \scalebox{\bt@scalefactor}{%
        \lower\ht0\hbox{$#1\bullet\m@th$}%
      }%
    }%
    \hfil
  }\hspace{.07em}
}
\makeatother

\newcommand{\enger}{\setlength{\arraycolsep}{.8pt}
                           \renewcommand{\arraystretch}{0.7} }

\newcommand{\matze}[2]{\enger{
\left(
\begin{array}{c}
\scriptstyle #1 \\
\scriptstyle #2 \\
\end{array}
\right)
}}
\newcommand{\matez}[2]{\enger{
\left(
\begin{array}{rr}
\scriptstyle #1\; & \; \scriptstyle #2\; \\
\end{array}
\right)
}}

\title{A triangulated hull and a Nakayama closure of the stable module category inside the homotopy category}

\subjclass[2020]{16G10, 16E05, 18G65, 18G80}
\keywords{stable module category, homotopy category, triangulated category, stable Grothendieck group, stable equivalence of Morita type}

\author{Sebastian Nitsche}
\address{Sebastian Nitsche, Institute of Algebra and Number Theory, University of Stuttgart,\; Pfaffenwaldring 57, 70569 Stuttgart, Germany}
\email{sebastian.nitsche@mathematik.uni-stuttgart.de}

\begin{document}

\begin{abstract}
The stable module category has been realized as a subcategory of the unbounded homotopy category of 
projective modules by Kato. We construct the triangulated hull of this subcategory inside the homotopy category.
This can also be used to characterize self-injective algebras.
Moreover, we extend this construction to a subcategory closed under an induced Nakayama functor.
Both of these categories are shown to be preserved by stable equivalences of Morita type.
As an application, we study the Grothendieck group of this triangulated hull and compare it with the
stable Grothendieck group.
\end{abstract}

\maketitle

\section*{Introduction}
When studying the representation theory of finite dimensional algebras,
one often considers three associated categories.
The abelian category of finitely generated modules, the derived module category which is a triangulated category
and the stable module category.
Over a self-injective algebra, the stable module category also has a triangulated structure.
However, in general, the stable module category is neither abelian nor triangulated.

In \cite{Kato_Kernels}, Kato gives an equivalence $\stmod A \to \LL_A$ 
between the stable module category and a full subcategory $\LL_A$
of the unbounded homotopy category of projective modules $\KKp A$.
This category $\LL_A$ can be enlarged to a triangulated category inside $\KKp A$.
We define two triangulated subcategories of $\KKp A$ via orthogonality conditions 
such that they both have $\LL_A$ as a subcategory; cf.\ \cref{Def:Category_H}.
The category $\HP A$ has objects which do not have non-zero morphisms to bounded complexes of projective-injective modules.
The category $\Hstp A$ has objects which do not have non-zero morphisms to bounded complexes of strongly projective-injective modules.
In summary, we will obtain the following chain of subcategories. 
With the exception of $\stmod A \simeq \LL_A$, all of these are triangulated categories.
\[
\begin{tikzcd}[column sep = .7cm]
\LL_A \ar[r, hookrightarrow] & \HP A \ar[r, hookrightarrow] & \Hstp A \ar[r, hookrightarrow] & \KKp A \\
\stmod A \ar[u, "\sim" sloped]
\end{tikzcd}
\]
While $\Hstp A$ is a larger category than $\HP A$, it is closed under a functor $\nu_{\mathcal{K}}$ which is
induced by the Nakayama functor $\nu_A$; cf.\ \cref{Def:nu_K}.
If $A$ has finite global dimension, 
$\nu_{\mathcal{K}}$ is equivalent to the derived Nakayama functor between the bounded homotopy categories 
$\KKbp A \to \KKbp A$. Our result is as follows.

\begin{LetterTheorem}[\cref{Thm:Triangulated-Hull-of-L}, \cref{Thm:Triangulated-Hull-of-L_Closed-under-nu}] 
\label{IntThm:TriangulatedCat}
Let $A$ be a finite dimensional algebra.
\begin{itemize}
\item[(1)] The category $\HP A$ is the smallest triangulated subcategory of $\KKp{A}$ that contains $\LL_A$
           and is closed under isomorphisms.

\item[(2)] The category $\Hstp A$ is the smallest triangulated subcategory of 
           $\KKp{A}$ that contains $\LL_A$ and is closed under $\nu_{\mathcal{K}}$ and under isomorphisms.
\end{itemize}
\end{LetterTheorem}
As an application of the first result, we discuss the Grothendieck group of the triangulated category $\HP A$.
We define an alternative Grothendieck group $\mathrm{G}_0^\PP(A)$ of $\stmod A$ via perfect exact sequences 
instead of short exact sequences; cf.\ \cref{Def:Grothendieck_PerfSeq}.
Thereby, a short exact sequence is said to be perfect exact if the induced sequence under the functor 
$\Hom_A(-,A)$ is exact as well. 
\begin{LetterTheorem}[\cref{Thm:GrothendieckGroup_stmod-HP}]
The equivalence $\stmod A \to \LL_A$ induces an isomorphism 
\[
G_0^\PP(A) \simeq G_0(\HP A).
\]
\end{LetterTheorem}
In contrast to the known stable Grothendieck group, $\mathrm{G}_0^\PP(A)$ can be non-zero even for algebras 
of finite global dimension. We additionally see, that $\mathrm{G}_0^\PP(A)$ is preserved by 
stable equivalences that preserve perfect exact sequences up to projective direct summands; 
cf.\ \cref{Thm:InducedIsoOnGrothendieckGroup}.

Regarding the second result of \cref{IntThm:TriangulatedCat}, a result by Fang, Hu and Koenig found in 
\cite[Theorem 4.3]{FangHuKoenig_DerivedEquiv} implies that $\Hstp A$ is a characteristic subcategory of 
$\KKbp A$ if $A$ has finite global dimension and 
$\nu$-dominant dimension at least $1$; cf.\ \cref{Cor:Kstp-characteristic_finite-gdim}. 
We also provide a variation of this consequence for algebras of arbitrary global dimension; cf.\ \cref{Thm:Kstp-characteristic_in_H}.

Additionally, we show that a stable equivalence of Morita type preserves the categories $\HP A$ and $\Hstp A$.
\begin{LetterTheorem}[\cref{Thm:InducedEquivalenceOnLAndH}] \label{IntThm:InvariantStEquivMoritaType}
Suppose ${}_A M_B$ and ${}_B N_A$ are bimodules that induce a stable equivalence of Morita type 
such that $M$ and $N$ do not have any non-zero projective bimodule as direct summand.
\begin{itemize}
\item[(1)] Applying $-\otimes_A M$ componentwise induces an equivalence of categories $\LL_A \rightarrow \LL_B$.

           \noindent
           If $A$ and $B$ are self-injective, this is an equivalence of triangulated categories.
           
\item[(2)] Applying $-\otimes_A M$ componentwise induces an equivalence of triangulated categories 
           \[ \HP A \rightarrow \HP B. \]

\item[(3)] Applying $-\otimes_A M$ componentwise induces an equivalence of triangulated categories 
           \[ \Hstp A \rightarrow \Hstp B. \]
\end{itemize}
\end{LetterTheorem}

In case that $A$ is a self-injective algebra, $\stmod A$ is already a triangulated category.
Therefore, $\LL_A$ is a triangulated subcategory of $\KKp A$ and all previously mentioned subcategories
coincide. In this sense, the above constructions are compatible with the existing structure
of $\stmod A$ and $\KKp A$. We show, that this characterizes the property of $A$ to be self-injective.
In the following, $\KKptac A$ denotes the full subcategory of totally acyclic complexes in $\KKp A$.
\begin{LetterTheorem}[\cref{Thm:Charact_SelfInjective}] \label{IntThm:Self-Injective}
The following are equivalent for a finite dimensional algebra $A$.
\begin{itemize}
\item[(1)] $A$ is self-injective.

\item[(2)] $\LL_A$ is a triangulated subcategory of $\KKp A$.

\item[(3)] $\LL_A = \HP A$.

\item[(4)] $\LL_A$ is closed under taking shifts in $\KKp A$.

\item[(5)] $\LL_A = \KKptac A$.
\end{itemize}
If one of the above conditions holds, the functor $\stmod A \to \LL_A$ is an equivalence of triangulated categories.
Furthermore, we have $\KKptac A = \LL_A = \HP A = \Hstp A$.
\end{LetterTheorem} 

This article is structured as follows. In a first section, we recall some necessary notation and the construction
of the category $\LL_A$. Furthermore, we introduce the definition of the categories $\HP A$ and $\Hstp A$ in \cref{Def:Category_H}.
In \cref{Sec:Category_HP(proj A)}, we show that $\HP A$ is the smallest triangulated subcategory of $\KKp{A}$ 
that contains $\LL_A$ and is closed under isomorphisms. The following \cref{Sec:Grothendieck} studies the Grothendieck
group of $\HP A$. In \cref{Sec:Hstp}, we show that $\Hstp A$ is the smallest triangulated subcategory of 
$\KKp{A}$ that contains $\LL_A$ and is closed under $\nu_{\mathcal{K}}$ and under isomorphisms.
\cref{Sec:SelfInjective} is dedicated to the special case of self-injective algebras. Finally, we show that a stable
equivalence preserves the categories $\HP A$ and $\Hstp A$ in \cref{Sec:StEquivMoritaType}.

\section{Preliminaries}

Let $A$ and $B$ be finite dimensional $k$-algebras over a field $k$.
We assume that $A$ and $B$ do not have any semisimple direct summands.

Let $\rmod A$ be the category of finitely generated right $A$-modules.
We write $\proj A$ and $\inj A$ for the full subcategory of projective and injective $A$-modules respectively.
The corresponding categories of left $A$-modules are denoted by $\modleft A$, $\projleft A$ and $A$-inj.
If not specified otherwise, modules are finitely generated right modules.

Given morphisms $f : X \to Y$ and $ g: Y \to Z$ in a category, we denote the composite of $f$ and $g$ by $fg : X \to Z$.
On the other hand, we write $\GG \circ \FF$ for the composite of two functors
$\FF : \mathcal{C} \to \mathcal{D}$ and $\GG : \mathcal{D} \to \mathcal{E}$ between categories.

Let $\mathcal{A}$ be an additive category and $\mathcal{S}$ be a full subcategory of $\mathcal{A}$.
We write $\!{}^\perp\mathcal{S}$ for the full subcategory of $\mathcal{A}$ consisting of all
objects $X$ in $\mathcal{A}$ such that $\Hom_{\mathcal{A}}(X,Z) = 0$ for all $Z \in \mathcal{S}$.
Analogously, we define the category $\mathcal{S}^\perp$.

We write $\Du(-) := \Hom_k(-,k) : \rmod A \to \modleft A$ for the $k$-duality.
Recall that the functor $(-)^\ast := \Hom_A(-,A) : \rmod A \to \modleft A$ restricts to an equivalence
$(-)^\ast : \proj A \to \projleft A$. 
Similarly, the Nakayama functor $\nu_A := \Du \Hom_A(-,A) : \rmod A \to \rmod A$ restricts to an equivalence
$\nu_A : \proj A \to \inj A$. The quasi-inverse is given by $\nu_A^{-1} = (\Du(-))^\ast$.

We denote the full category of projective-injective $A$-modules by $\PP_A$.
We say that an $A$-module $Z$ is \textit{strongly projective-injective} if $\nu^k Z$ is projective 
for all $k \in \ZZ_{\geqslant 0}$. The full subcategory of strongly projective-injective $A$-modules will be denoted
by $\stp A$. Note that strongly projective-injective modules are projective-injective.
The following variation of dominant dimension has been introduced in \cite[Definition 2.4]{FangHuKoenig_DerivedEquiv}.
\begin{Definition} \label{Def:DomDim}
Let $0 \to A \to I^0 \to I^1 \to I^2 \to \cdots$ be a minimal injective resolution of $A$.
The $\nu$\textit{-dominant dimension} of $A$ is defined as the largest $d \in \ZZ_{\geqslant 0}$ such that
$I^k$ is strongly projective-injective for all $k < d$. We set $d = \infty$ if $I^k$ is 
strongly projective-injective for all $k \geq 0$.
We denote the $\nu$-dominant dimension of $A$ by $\nddim A$.
\end{Definition}

The \textit{stable module category} $\stmod A$ is the category with the same objects as $\rmod A$
and with morphisms $\underline{\Hom}_A(X,Y) := \Hom_A(X,Y)/\mathrm{PHom}_A(X,Y)$ for $X,Y \in \rmod A$.
A morphism $f : X \to Y$ is an element of $\mathrm{PHom}_A(X,Y)$ if there exists a $P \in \proj A$ 
such that $f$ factors through $P$. We say that $A$ and $B$ are stably equivalent, if there exists an equivalence
$\stmod A \to \stmod B$.

An element $F^\bt = (F^k)_{k \in \ZZ} \in \CC A$ in the category of complexes 
will be written as a cochain complex with differential 
$(d^k)_{k \in \ZZ} := (d_F^k)_{k \in \ZZ}$ as follows.
\[
\cdots \rightarrow F^{-2} \xrightarrow{d^{-2}} F^{-1} \xrightarrow{d^{-1}} F^0 \xrightarrow{d^0} F^1 \xrightarrow{d^1} F^2 \rightarrow \cdots
\]
We denote the cohomology of $F^\bt$ in degree $k \in \ZZ$ by $\HH^k(F^\bt) := \ker d^k/\im d^{k-1}$.
For $n \in \ZZ$, truncation $\tau_{\leqslant n}\, F^\bt$ of $F^\bt$ is defined as follows.
\[
\cdots \rightarrow F^{n-2} \xrightarrow{d^{n-2}} F^{n-1} \xrightarrow{d^{n-1}} F^n \rightarrow 0 \rightarrow 0 \rightarrow \cdots
\]
We often abbreviate $F^{\leqslant n} := \tau_{\leqslant n}\, F^\bt$ and similarly for 
$F^{\geqslant n} := \tau_{\geqslant n}\, F^\bt$.
We also use the notation $F^{\leqslant n}$ to indicate that $F^k = 0$ for $k > n$.
An $A$-module $X$ will be identified with the complex $X \in \CC{A}$ consisting of $X$ concentrated in degree zero.

By componentwise application, the equivalence $(-)^\ast = \Hom_A(-,A) : \proj A \to A$-proj can be extended to an equivalence 
\[
(-)^\ast : \CCp A \to \CCpleft A : F^\bt \to F_\bt^\ast = F^{\bt, \ast}
\]
Here, we write $F_k^\ast := F^{k,\ast} := \left(F^k\right)^\ast$ as the chain complex with
$d^{F^\ast}_k := d_F^{k,\ast} := \left(d_F^k\right)^\ast$ for $k \in \ZZ$.
\[
\cdots \rightarrow F_2^\ast \xrightarrow{d_1^\ast} F_1^\ast \xrightarrow{d_0^\ast} F_0^\ast
		 \xrightarrow{d_{-1}^\ast} F_{-1}^\ast \xrightarrow{d_{-2}^\ast} F_{-2}^\ast \rightarrow \cdots
\]
We denote the homology of $F_\bt^\ast$ in degree $k \in \ZZ$ by $\HH_k(F_\bt^\ast) = \ker(d_{k-1}^\ast)/\im(d_k^\ast)$.
In this sense, we use both chain complexes and cochain complexes in our notation.
However, we reserve the notation of chain complexes for dualized cochain complexes.

Similarly to $(-)^\ast$, the functors $\Du$ and $\nu$ also induce equivalences
$\Du : \CC A \to \CCleft A$ and $\nu : \CCp A \to \CCi A$ respectively.

Let $\mathcal{A}$ be an additive subcategory of $\rmod A$.
We write $\mathcal{K}(\mathcal{A})$ for the homotopy category of complexes over $\mathcal{A}$.
If $\mathcal{A}$ is an abelian category, $\mathcal{D}(\mathcal{A})$ denotes the derived category.
We write $\mathcal{C}^+(\mathcal{A})$, $\mathcal{K}^+(\mathcal{A})$ and $\mathcal{D}^+(\mathcal{A})$ for the subcategory consisting of 
left bounded complexes in $\mathcal{C}(\mathcal{A})$, $\mathcal{K}(\mathcal{A})$ and $\mathcal{D}(\mathcal{A})$ respectively. 
Similarly, we write $\mathcal{C}^-(\mathcal{A})$, $\mathcal{K}^-(\mathcal{A})$ and $\mathcal{D}^-(\mathcal{A})$ for right bounded complexes.
The subcategory of left and right bounded complexes is denoted by 
$\mathcal{C}^b(\mathcal{A})$, $\mathcal{K}^b(\mathcal{A})$ or $\mathcal{D}^b(\mathcal{A})$.
By $\mathcal{C}^{+,b}(\mathcal{A})$, $\mathcal{K}^{+,b}(\mathcal{A})$ and $\mathcal{D}^{+,b}(\mathcal{A})$ 
we denote the subcategory of left bounded complexes that are bounded in cohomology.	
Similarly, the subcategories $\mathcal{C}^{+,b^\ast}\!(\proj A)$, $\mathcal{K}^{+,b^\ast}\!(\proj A)$ and 
$\mathcal{D}^{+,b^\ast}\!(\proj A)$ consist of the left bounded complexes 
$F^\bt$ with bounded homology $\HH_\bt(F_\bt^\ast)$.
The analogue categories for right bounded complexes are defined similarly.
Finally, we write $\KKptac A$ for the full subcategory of $\KKp A$ consisting of totally acyclic complexes.

In \cite{Kato_Kernels}, Kato constructs an equivalence $\stmod A \to \LL_A$ between the stable module category and
a full subcategory $\LL_A$ of $\KKp A$.
\begin{Definition}[Kato] \label{Def:L-Category}
Let $\LL_A$ be the full subcategory of $\KKp{A}$ defined as follows.
\[
\LL_A = \{ F^\bt \in \KKp{A} \;\vert\; \mathrm{H}^{<0}(F^\bt) = 0, \, \mathrm{H}_{\geqslant0}(F^\ast_\bt) = 0 \}
\]
\end{Definition}

Let $X \in \rmod A$. We sketch the construction of a complex $F_X^{\bt} \in \LL_A$.
See \cite[Lemma 2.9]{Kato_Mono} for more details.

Choose $\tleq F_X^\bt$ as the minimal projective resolution of $X$
and complete $F_X^{0,\ast} \to F_X^{-1,\ast}$ to the minimal projective resolution $\tau_{\geqslant -1} F_X^{\bt, \ast}$ 
of $\Tr(X)$. Splicing these two projective resolutions together gives a complex $F_X^{\bt} \in \LL_A$. 
In particular, $\tau_{\geqslant 1} F_X^{\bt, \ast}$ is the minimal projective resolution of $X^\ast$. 
Moreover, $F_X^0$ is the projective cover of $X$ and $F_X^{1,\ast}$ the projective cover of $X^\ast$.
If $X$ is simple, $\nu(F_X^0)$ is the injective hull of $X$.

For $X \xrightarrow{f} Y$ in $\rmod A$, we lift $f$ to a morphism between projective resolutions 
$\tleq F_X^\bt \to \tleq F_Y^\bt$ in non-negative degrees. Similarly, we lift $f^\ast$ to a morphism
$\tau_{\geqslant -1} F_X^{\bt, \ast} \to \tau_{\geqslant -1} F_X^{\bt, \ast}$ such that the lifts coincide in degree $-1$
and $0$. Together, this gives a a morphism $f^\bt : F_X^\bt \to F_Y^\bt$ in $\LL_A$.

In \cite[Theorem 2.6]{Kato_Kernels} it was shown that this defines an equivalence $\stmod A \to \LL_A$ 
in the setting of commutative rings. However, the proof still works in the same way for arbitrary finite dimensional algebras.
In the future, we will often use this equivalence without further comment.
\begin{Theorem}[Kato] \label{Thm:Equivalence_F}
The mapping $X \mapsto F_X^\bt$ defines a functor $\FF : \stmod A \to \KKp A$. 
The functor $\FF$ restricts to an equivalence
\[
\FF : \stmod A \xrightarrow{\sim} \LL_A
\]
with quasi-inverse $\HH^0(\tleq(-)) : \LL_A \to \stmod A$.
\end{Theorem}

We will also need the following class of short exact sequences.
\begin{Definition} \label{Def:PerfSeq}
A short exact sequence $0 \to X \to Y \to Z \to 0$ in $\rmod A$ is called \textit{perfect exact} if
the induced sequence $0 \to X^\ast \to Y^\ast \to Z^\ast \to 0$ is exact in $A$-mod.
\end{Definition}

Under the equivalence $\FF: \stmod A \to \LL_A$, perfect exact sequences correspond to distinguished triangles
in $\KKp A$ and vice versa. The following proposition is based on \cite[Proposition 3.6]{Kato_Mono}.
\begin{Proposition}[Kato] \label{Prop:PerfSeq-DistTriang}
Suppose given a sequence $X \xrightarrow{f} Y \xrightarrow{g} Z$ in $\rmod A$ such that 
$Y$ has no projective direct summand.
The following are equivalent.
\begin{itemize}
\item[(1)] There exists a projective module $P$ and morphisms $p$ and $q$ such that
\[ 
0 \rightarrow X \xrightarrow{\matez{f}{p}} Y \oplus P \xrightarrow{\matze{g}{q}} Z \rightarrow 0
\]
is a perfect exact sequence in $\rmod A$.

\item[(2)] The sequence 
$ F_X^\bt \xrightarrow{f^\bt} F_Y^\bt \xrightarrow{g^\bt} F_Z^\bt \to$
is a distinguished triangle in $\KKp A$.
\end{itemize}
\end{Proposition}

Recall that $\PP_A$ denotes the category of projective-injective $A$-modules.
The subcategory of strongly projective-injective $A$-modules is denoted by $\stp A$.
We introduce the following subcategories of $\KKp A$ with a left bound on cohomology and a right bound on homology.

\begin{Definition} \label{Def:Category_H}
We denote by $\Hp A$ the full subcategory of $\KKp A$ consisting of all complexes
$F^\bt \in \KKp A$ such that there exist $l, r \in \ZZ$ with $\HH^{< l}(F^\bt) = 0$ and 
$\HH_{\geqslant r}(F_\bt^{\ast}) = 0$.

We denote by \;$\HP A$\, the full subcategory of $\Hp A$ consisting of all complexes in
$\KPperp A = \{ F^\bt \in \KKp A : \Hom_{\KKp A}(F^\bt, Z^\bt) = 0 \text{ for all } Z^\bt \in \KP A \}$.

We denote by $\Hstp A$ the full subcategory of $\Hp A$ consisting of all complexes in
$\Kstperp A = \{ F^\bt \in \KKp A : \Hom_{\KKp A}(F^\bt, Z^\bt) = 0 \text{ for all } Z^\bt \in \Kstp A \}$.
\end{Definition}

\begin{Remark} \label{Rem:Properties_H-HP-Hstp}
\begin{itemize}
\item[(1)] Note that $\Hp A$ is a triangulated subcategory of $\KKp A$.
           As perpendicular categories, $\HP A$ and $\Hstp A$ are 
           triangulated subcategories of $\Hp{A}$.
           Furthermore, $\Hp A$, $\HP A$ and $\Hstp A$ are closed under isomorphisms in $\KKp A$.

\item[(2)] We have a chain of subcategories $\HP A \subseteq \Hstp A \subseteq \Hp A \subseteq \KKp A$.
           Furthermore, $\LL_A \subseteq \Hp A$ by letting $l = 0$ and $r = 0$ in the definition of $\Hp A$.
           In this sense, the boundary conditions of $\Hp A$ can be seen as a weaker version of those in $\LL_A$.
           They will be used in \cref{Lem:Reduction-To-ProjRes}.
           In particular, the smallest triangulated subcategory of $\KKp A$ that contains $\LL_A$ 
           must be contained in $\Hp A$.

\item[(3)] In general, complexes in $\Hp A$ are neither left bounded nor right bounded.          
           However, we have $\Hp A \simeq \KKbp{A}$ if and only if $\gdim{A} < \infty$. 
           \vspace*{2mm}           
           
           \noindent
           In fact, every complex in $\Hp A$ can be truncated on the right to obtain a projective resolution in $\rmod A$.
           Thus, the complex must split eventually, if it is unbounded on the left and $\gdim A < \infty$.
           Similarly, every complex in $\Hp A$ can be truncated on the left to obtain a projective resolution in $A$-mod
           after applying $(-)^\ast$. Moreover, every projective resolution of a left or right $A$-module occurs in this way.
           
\item[(4)] Note that $\HP A = \Hstp A = \Hp A$ if $A$ has no projective-injective modules.
           In particular, we have that $\HP A = \Hstp A \simeq \KKbp A$ is the bounded derived category of $A$
           if additionally $\gdim A < \infty$.
           The same holds for $\Hstp A$ and $\Hp A$ if $A$ has no strongly projective-injective modules.
\end{itemize}
\end{Remark}
The categories discussed in this article can be visualized as follows.
Note that the inclusion $\LL_A \hookrightarrow \HP A$ will be verified in \cref{Lem:L-is-in-K_perp}.
In general, this chain of subcategories has a proper inclusion at every position.
However, we will see later in \cref{Thm:Charact_SelfInjective} that 
$\KKptac A = \LL_A = \HP A = \Hstp A$ if and only if $A$ is self-injective.    
\[
\begin{tikzcd}[column sep = .6cm]
\!\!\KKptac A \ar[r, hookrightarrow] & \LL_A \ar[r, hookrightarrow] & \HP A \ar[r, hookrightarrow] & \Hstp A \ar[r, hookrightarrow] & \Hp A \ar[r, hookrightarrow] & \KKp A \\
                                 & \stmod A \ar[u, "\mathcal{F}"', "\sim" sloped]
\end{tikzcd}
\]
For now, we discuss $\HP A$ in more detail in the next section.
Later, in \cref{Sec:Hstp}, we return to the category $\Hstp A$.

\section{\texorpdfstring{A triangulated hull}{A triangulated hull}}
\label{Sec:Category_HP(proj A)}

The aim of this section is to show that $\HP A$ is the smallest triangulated subcategory of $\KKp A$ that contains $\LL_A$
and is closed under isomorphisms.
In order to prove this, we have to verify that $\LL_A$ is contained in $\HP A$ and that a complex $F^\bt \in \HP A$ 
is an element of any triangulated subcategory of $\KKp A$ that contains $\LL_A$ and is closed under isomorphisms. 
The first assertion follows from the next two results. For the second assertion, we then proceed as follows.

Initially, we observe that a complex is in $\KPperp A$ if and only if its cohomology is in $\Pperp_A$; cf.\ \cref{Lem:H-iff-Complex_in-perp}.
Next, we reduce the problem in \cref{Lem:Reduction-To-ProjRes} to projective resolutions of modules in $\Pperp_A$.
As a further reduction step, we see in \cref{Lem:ProjRes-if-CompFactors-in-T} that it is enough to consider simple modules in $\Pperp_A$.
Finally, we show in \cref{Lem:ProjRes-in-T_Simple} that the assertion holds for projective resolutions of simple modules in $\Pperp_A$.

We start with the following lemma. In an exact degree, a complex in $\KK A$ has no non-zero morphism to a projective
module or from an injective module. The same holds for the dual complex in $\KKpleft A$.
\begin{Lemma} \label{Lem:Zero-Hom_At-exact-Pos}
Let $F^\bt \in \KK{A}$ and $k \in \ZZ$.
\begin{itemize}
\item[(1)] If $\HH^k(F^\bt) = 0$ then $\Hom_{\KK{A}}(F^\bt, Z[-k]) = 0$ for $Z \in \inj A$.
\item[(2)] If $\HH^k(F^\bt) = 0$ then $\Hom_{\KK{A}}(Z[-k], F^\bt) = 0$ for $Z \in \proj A$.
\end{itemize}
Now, assume that $F^\bt \in \KKp A$.
\begin{itemize}
\item[(1')] If $\HH_k(F_\bt^\ast) = 0$ then $\Hom_{\KKp{A}}(Z[-k],F^\bt) = 0$ 
            for $Z \in \proj A$ with $Z^\ast \in A\text{-}\mathrm{inj}$.
\item[(2')] If $\HH_k(F_\bt^\ast) = 0$ then $\Hom_{\KKp{A}}(F^\bt, Z[-k]) = 0$ for $Z \in \proj A$.
\end{itemize}
\end{Lemma}
\begin{proof} 
\textit{Ad (1) and (2).}
Suppose given a morphism of complexes $f^\bt : F^{\bt} \to Z[-k]$. In particular, we have $d^{k-1} f^k = 0$.
\[
\begin{tikzcd}
\cdots \ar[r] & F^{k-1} \ar[r, "d^{k-1}"] \ar[d, "f^{k-1}"] & F^{k} \ar[r, "d^k"] \ar[d, "f^{k}"] & 
		F^{k+1} \ar[r]\ar[d, "f^{k+1}"] & \cdots & F^\bt \ar[d, "f^\bt"] \\
\cdots \ar[r] & 0 \ar[r] & Z \ar[r] & 0 \ar[r] & \cdots & Z[-k]
\end{tikzcd}
\]
By assumption, we have $\ker{d^k} = \im{d^{k-1}} \subseteq \ker{f^k}$.
Thus, there exists a morphism $g : F^{k} / \, \ker{d^k} \to Z$ such that the following diagram commutes.
\[
\begin{tikzcd}
F^{k} \ar[dr, "\pi"] \ar[ddr, "f^k", bend right] \ar[rrd, "d^k", bend left] \\
 & F^{k} / \, \ker{d^k} \ar[r, hookrightarrow, "\tilde{d^k}"] \ar[d, "g"] & F^{k+1} \\
 & Z
\end{tikzcd}
\]
Since $Z$ is injective, there exists a morphism $h : F^{k+1} \to Z$ with $\tilde{d^k} \, h = g$.
We obtain
\[
d^k \, h = \pi \, \tilde{d^k} \, h = \pi \, g = f^k
\]
so that $f^\bt = 0$.
This shows part (1). Part (2) follows dually.

\textit{Ad (1') and (2').}
Since $\Hom_{\KKp{A}}(Z[-k],F^\bt) \simeq \Hom_{\KKpleft{A}}(F_\bt^\ast, Z^\ast[-k])$, part $(1')$ follows from part $(1)$
applied to left $A$-modules.

Similarly, since $\Hom_{\KKp{A}}(F^\bt, Z[-k]) \simeq \Hom_{\KKpleft{A}}(Z^\ast[-k], F_\bt^\ast)$, 
part $(2')$ follows from part $(2)$ applied to left $A$-modules.
\end{proof}

We extend the previous result to morphisms between complexes. 
In particular, this lemma shows that $\LL_A$ is contained in $\HP A$.
\begin{Lemma} \label{Lem:L-is-in-K_perp} 
Let $F^\bt \in \KKp A$.
\begin{itemize}
\item[(1)] If $\HH^\bt(F^\bt) = 0$, then $\Hom_{\KK A}(F^\bt, Z^\bt) = 0$ for all $Z^\bt \in \KKbi A$.
\item[(2)] If $\HH_\bt(F_\bt^\ast) = 0$, then $\Hom_{\KKp A}(F^\bt, Z^\bt) = 0$ for all $Z^\bt \in \KKbp A$.
\item[(1')] If $\HH_\bt(F_\bt^\ast) = 0$, then $\Hom_{\KKp A}(Z^\bt, F^\bt) = 0$ for all $Z^\bt \in \KKbp A$ with
            ${Z^\ast_\bt \in \mathrm{K}^b(A\text{-}\inj)}$.
\item[(2')] If $\HH^\bt(F^\bt) = 0$, then $\Hom_{\KKp A}(Z^\bt, F^\bt) = 0$ for all $Z^\bt \in \KKbp A$.
\item[(3)] If $F^\bt \in \LL_A$, then $F^\bt \in \HP A \,$.
\end{itemize}
\end{Lemma}
\begin{proof}
At first we show the following claim.

\textit{Claim.} Let $Z^\bt \in \KKb A$.
We have $\Hom_{\KK A}(F^\bt, Z^\bt) = 0$ if $\Hom_{\KK A}(F^\bt, Z^k) = 0$ for all $k \in \ZZ$.

Suppose that $Z^\bt$ is non-zero. Since $Z^\bt$ is bounded, there exists an $l \in \ZZ$ such that $Z^l \neq 0$ and 
$Z^k = 0$ for $k < l$.
Moreover, there exists an $r \in \ZZ$ such that $Z^r \neq 0$ and $Z^k = 0$ for $k > r$. 

We proceed by induction on the number of non-zero terms of $Z^\bt$.

Suppose that $l = r$. Then $Z^\bt = Z^l$ and $\mathrm{Hom}_{\KK{A}}(F^\bt, Z^l) = 0$ by assumption.

Suppose that $l < r$. By induction, we can assume that $\Hom_{\KK{A}}(F^\bt, Z^{\leqslant r-1}) = 0$.
Thus, there exist homotopy maps $h^{k} : {F^{k} \to Z^{k-1}}$ for $k \leq r$ such that
\[
h^k \, d_Z^{k-1} + d_F^k \, h^{k+1} = f^k, \quad \mathrm{ for }\; k < r.
\]
Consider the following diagram.
\[
\begin{tikzcd}[row sep=1.5cm, column sep=1.5cm]
\cdots \ar[r] & F^{r-2} \ar[r, "d_F^{\,r-2}"] \ar[d, "f^{r-2}"] & 
				F^{r-1} \ar[r, "d_F^{\,r-1}"] \ar[d, "f^{r-1}"] \ar[dl, "h^{r-1\;\;}" above, near start] & 
				F^{r} \ar[r, "d_F^{\,r}"] \ar[d, "f^{r}"] \ar[dl, "h^r" above, near start] & 
				F^{r+1} \ar[r]\ar[d] \ar[dl, dashed, "h^{r+1\;\;}" above, near start] & 
					\cdots & F^\bt \ar[d, "f^\bt"] \\
\cdots \ar[r] & Z^{r-2} \ar[r, "d_Z^{\,r-2}"] & Z^{r-1} \ar[r, "d_Z^{\,r-1}"]  & Z^{r} \ar[r, "d_Z^{\,r}"] & 
		0 \ar[r] & \cdots & Z^\bt
\end{tikzcd}
\]
Note that
\[
d_F^{\,r-1} \, \left( f^r - h^r \, d_Z^{\,r-1} \right) 
	= f^{r-1} \, d_Z^{\,r-1} - f^{r-1} \, d_Z^{\,r-1} + h^{r-1} \, d_Z^{\,r-2} \, d_Z^{\,r-1} = 0
\]
so that $f^r - h^r \, d_Z^{\,r-1}$ induces a morphism of complexes $F^\bt \rightarrow Z^r$.
However, $\Hom_{\KK{A}}(F^\bt, Z^r) = 0$ by assumption.
This yields a homotopy map $h^{r+1} : F^{r+1} \to Z^r$ with 
\[
d_F^{\,r} \, h^{r+1} = f^r - h^r \, d_Z^{\,r-1}
\; \Leftrightarrow \;
d_F^{\,r} \, h^{r+1} + h^r \, d_Z^{\,r-1} = f^r.
\]
In conclusion, we obtain $\Hom_{\KK{A}}(F^\bt, Z^\bt) = 0$.

By \cref{Lem:Zero-Hom_At-exact-Pos}.$(1,2')$, we have $\Hom_{\KK A}(F^\bt, Z^k) = 0$ for $k \in \ZZ$ 
in the situation of part $(1)$ and $(2)$ respectively. Hence, part (1) and (2) follow from the claim above.
Note that for $Z^\bt \in \KKp A$, we have $\Hom_{\KKp A}(Z^\bt, F^\bt) = 0$ if and only if 
$\Hom_{\KKpleft{A}}(F_\bt^\ast, Z_\bt^\ast) = 0$.
Thus, $(1')$ and $(2')$ follow from the versions of $(1)$ and $(2)$ for left $A$-modules respectively. 

Finally, let $F^\bt \in \LL_A$ and $Z^\bt \in \KP A$. By definition, we have $\HH^k(F^\bt) = 0$ for $k < 0$ and 
$\HH_k(F_\bt^\ast) = 0$ for $k \geq 0$. In particular, $F^\bt \in \Hp A$.
Furthermore, \cref{Lem:Zero-Hom_At-exact-Pos}.$(1)$ shows $\Hom_{\KKp A}(F^\bt, Z^k) = 0$ for $k < 0$ and
\cref{Lem:Zero-Hom_At-exact-Pos}.$(2')$ shows that $\Hom_{\KKp A}(F^\bt, Z^k) = 0$ for $k \geq 0$.
Now, the claim above gives $\Hom_{\KKp A}(F^\bt, Z^\bt) = 0$ so that $F^\bt \in \KPperp A$.
Together, we obtain $F^\bt \in \HP A$ which shows part (3).
\end{proof}

The next several results aim to show that a complex $F^\bt \in \HP A$ 
is contained in any triangulated subcategory of $\KKp A$ that contains $\LL_A$ and is closed under isomorphisms.
We start with the following important observation about complexes in $\HP A$ and $\Hstp A$.

\begin{Lemma} \label{Lem:H-iff-Complex_in-perp}
Let $F^\bt \in \KK{A}$.
\begin{itemize}
\item[(1)] $F^\bt \in \KPperp A$ if and only if $\HH^k(F^\bt) \in \Pperp_A$ for all $k \in \ZZ$.
\item[(2)] $F^\bt \in \Kstperp A$ if and only if $\HH^k(F^\bt) \in \stperp A$ for all $k \in \ZZ$.
\item[(1')] $F^\bt \in \mathcal{K}^b(\nu^{-1}\PP_A)^\perp$ if and only if $\HH^k(F^\bt) \in (\nu^{-1}\PP_A)^{\perp}$ for all $k \in \ZZ$.
\item[(2')] $F^\bt \in \Kstp A^\perp$ if and only if $\HH^k(F^\bt) \in (\stp A)^\perp$ for all $k \in \ZZ$.
\end{itemize}
\end{Lemma}

\begin{proof} 
Let $\mathcal{I}$ be a full subcategory of $\inj A$.
We show that we have $F^\bt \in {}^\perp\mathcal{K}^b(\mathcal{I})$ if and only if 
$\HH^k(F^\bt) \in {}^\perp\mathcal{I}$ for all $k \in \ZZ$.
Letting $\mathcal{I} = \PP_A$ we obtain part (1) and letting $\mathcal{I} = \stp A$ we obtain part (2).

Suppose that $F^\bt \in {}^\perp\mathcal{K}^b(\mathcal{I})$. We fix a $k \in \ZZ$ with $\HH^k(F^\bt) \neq 0$. 
Let $Z \in \mathcal{I}$ and suppose given a morphism $f : \HH^k(F^\bt) \to Z$.

Consider the following commutative diagram. The morphism $\alpha$ exists since $Z$ is injective.
\[
\begin{tikzcd}[row sep=.9cm, column sep=.7cm]
\cdots \ar[r] & F^{k-1} \ar[r, "d^{k-1}"] &  F^k \ar[r, "d^k"] \ar[dr, twoheadrightarrow, "\pi"] & F^{k+1} \ar[r] & \cdots \\
 & \ker{d^{k}} \ar[ur, hookrightarrow, "\iota"] \ar[dr, twoheadrightarrow, "p"] & & F^k /\, \im{d^{k-1}} \ar[ddl, dashed, "\alpha"]\\
 & & \HH^k(F^\bt) \ar[ur, hookrightarrow] \ar[d, "f"] \\
 & & Z
\end{tikzcd}
\]
Since $d^{k-1} \, \pi \, \alpha = 0$, this yields a morphism of complexes $\pi \, \alpha : F^\bt \to Z[-k]$.
By assumption, there exists a homotopy map $h : F^{k+1} \to Z$ with $d^k \, h = \pi \, \alpha$.
We have
\[
0 = \iota \, d^k \, h = \iota \, \pi \, \alpha = p \, f
\]
so that $f = 0$.

Conversely, let $\HH^k(F^\bt) \in {}^\perp\mathcal{I}$ for all $k \in \ZZ$. 
Suppose given a morphism of complexes $F^\bt \xrightarrow{f^\bt} Z^\bt$ with $Z^\bt$ in $\mathcal{K}^b(\mathcal{I})$.
Let $r \in \ZZ$ be maximal such that $Z^r \neq 0$. By applying a shift $[-r]$ we may assume that $r = 0$.
We show that $f^\bt = 0$ by induction on the number of non-zero terms of $Z^\bt$.

Suppose that $Z^k = 0$ for $k \neq 0$. We have $d_F^{-1} f^0 = f^{-1}\, d_Z^{-1} = 0$ so that 
there exists a morphism $\alpha : F^0 /\, \im{d_F^{k-1}} \to H^0(F^\bt)$
with $f^0 = \pi\, \alpha$. This results in a morphism $g = i\,\alpha : \HH^0(F^\bt) \to Z^0$
such that the following diagram commutes.
\[
\begin{tikzcd}[row sep=.9cm, column sep=.7cm]
\cdots \ar[r] & F^{-1} \ar[r, "d_F^{-1}"] &  F^0 \ar[r, "d_F^0"] \ar[dr, twoheadrightarrow, "\pi"] \ar[ddd, bend right=35, "f^0" near start] & F^{1} \ar[r] & \cdots \\
 & \ker{d^{0}} \ar[ur, hookrightarrow, "\iota"] \ar[dr, twoheadrightarrow, crossing over, "p" near start] & & F^0 /\, \im{d_F^{-1}} \ar[ddl, dashed, "\alpha"]\\
 & & \HH^0(F^\bt) \ar[ur, hookrightarrow, "i"] \ar[d, dashed, "g" near start] \\
\cdots \ar[r] & 0 \ar[r, "d_Z^{-1}"] & Z^0 \ar[r, "d_Z^0"] & 0 \ar[r] & \cdots 
\end{tikzcd} 
\]
By assumption, $g = 0$. Hence $\iota \, f^0 = \iota \, \pi \, \alpha = p\,i\,\alpha =  p \, g = 0$.
This yields a morphism of complexes as follows.
\[
\begin{tikzcd}[row sep=.9cm, column sep=.7cm]
\cdots \ar[r] & 0 \ar[r] \ar[d] & \ker{d^0} \ar[r, "\iota"] \ar[d] &  F^0 \ar[r, "d^0"] \ar[d, "f^0"] & F^{1} \ar[r] \ar[d] & \cdots \\
\cdots \ar[r] & 0 \ar[r] & 0 \ar[r] & Z^0 \ar[r] & 0 \ar[r] & \cdots 
\end{tikzcd}
\]
Since $Z^0$ is injective, this morphism must be zero in $\KK A$ by \cref{Lem:Zero-Hom_At-exact-Pos}.(1)
so that there exists a morphism $h : F^1 \to Z^0$ with $d^0 h = f^0$.
This implies that $f^\bt : F^\bt \to Z^\bt$ is zero as well.

For the induction step, we consider the complex $\tau_{< 0} Z^\bt = Z^{<0}$. By induction hypothesis, we may assume that
${\Hom_{\KK{A}}(F^\bt, Z^{< 0}) = 0}$.
Hence there exist homotopy maps $h^k : F^k \to Z^{k-1}$ for $k \leq 0$ such that 
$h^{k-1} \, d_Z^{k-2} + d_F^{k-1} h^{k} = f^{k-1}$.
\[
\begin{tikzcd}[row sep=1.5cm, column sep=1.5cm]
\cdots \ar[r] & F^{-2} \ar[r, "d_F^{-2}"] \ar[d, "f^{-2}"] & 
				F^{-1} \ar[r, "d_F^{-1}"] \ar[d, "f^{-1}"] \ar[dl, "h^{-1}" above] & 
				F^{0} \ar[r, "d_F^0"] \ar[d, "f^{0}"] \ar[dl, "h^0" above] & 
				F^1 \ar[r,] \ar[d, "f^1"] & 
				\cdots  \\
\cdots \ar[r] & Z^{-2} \ar[r, "d_Z^{-2}"] & Z^{-1} \ar[r, "d_Z^{-1}"]  & Z^{0} \ar[r] & 
		0 \ar[r] & \cdots 
\end{tikzcd}
\]
Note that 
\[
d_F^{-1} \, \left( f^0 - h^0 \, d_Z^{-1} \right) 
	= f^{-1} \, d_Z^{-1} - f^{-1} \, d_Z ^{-1} + h^{-1} \, d_Z^{-2} \, d_Z^{-1} = 0
\]
so that $f^0 - h^0 \, d_Z^{-1}$ induces a morphism of complexes $F^\bt \rightarrow Z^0$.

Now we are in the same situation as above and we can conclude that there exists a morphism $F^1 \xrightarrow{h^1} Z^0$ such that $d_F^0 h^1 = f^0 - h^0 \, d_Z^{-1}$.
However, this yields $h^0 \, d_Z^{-1} + d_F^0 h^1 = f^0$. Thus $f^\bt = 0$.

It remains to show part (1') and part (2'). Let $\mathcal{C}$ be either $\nu^{-1}\PP_A$ or $\stp A$.
In both cases, $\mathcal{I} := \Du \mathcal{C}$ is a full subcategory of $A$-inj.
In particular, we can apply the arguments above for $\mathcal{I}$.

Suppose given $Z^\bt \in \mathcal{K}^b(\mathcal{C})$.
Note that $\Du Z^\bt \in \mathcal{K}^b(\mathcal{I})$ and
\[
\Hom_{\KK A}(Z^\bt,F^\bt) \simeq \Hom_{\KKleft A}(\Du F^\bt, \Du Z^\bt).
\]
Thus, the arguments above for left $A$-modules show that there exists a $Z^\bt \in \mathcal{K}^b(\mathcal{C})$
which satisfies $\Hom_{\KK A}(Z^\bt,F^\bt) \neq 0$ if and only if there is a ${}_A Z \in \mathcal{I} $ with 
$\Hom_{A}({}_A\HH^k(\Du F^\bt), {}_A Z) \neq 0$. Since $\Du(-)$ is exact, we have 
\[
\Hom_{A}({}_A \HH^k(\Du F^\bt), {}_A Z) \simeq \Hom_{A}({}_A(\Du\HH^k(F^\bt)), {}_A Z) \simeq \Hom_{A}((\Du Z)_A,\HH^k(F^\bt)_A).
\]
Using that $\Du Z \in \Du\mathcal{I} \simeq \mathcal{C}$, we are done.
\end{proof}

For a given complex $F^\bt \in \Hp A$, we want to construct a complex in $\LL_A$ which is related to
$F^\bt$ via distinguished triangles; cf.\ \cref{Lem:Reduction-To-ProjRes}.
This is done by removing non-zero cohomology of $F^\bt$ with projective resolutions.
Using the boundary conditions in the definition of $\Hp A$, there are only finitely many positions we have to consider
until we arrive at a complex in $\LL_A$.
The distinguished triangles that arise during the proof also give a way to calculate the class of $F^\bt$
in the Grothendieck group of $\HP A$.

Because it will be needed later, we first state the induction step in a more general lemma.
\begin{Lemma} \label{Lem:RemovingCohomologyOfF}
Suppose given $F^\bt \in \Hp A$.
Let $k \in \ZZ$ be minimal with $\HH^k(F^\bt) \neq 0$.
Let $H$ be a submodule of $\HH^k(F^\bt)$ with $P^\bt$ a projective resolution of $H$.
Then there exists a distinguished triangle
$P^\bt[-k] \to F^\bt \to C^\bt \to$
such that the following holds.
\begin{itemize}
\item[(1)] We have $\HH^j(C^\bt) = 0$ for $j < k$ and $\tau_{\geqslant k} C^\bt = \tau_{\geqslant k} F^\bt$.

\item[(2)] There is a short exact sequence $0 \to H \to \HH^k(F^\bt) \to \HH^k(C^\bt) \to 0$.
           In particular, we have $\HH^k(C^\bt) = 0$ if $H = \HH^k(F^\bt)$.
           
\item[(3)] Suppose that $F^\bt \in \LL_A$ with $k = 0$. Then we also have $C^\bt \in \LL_A$. 
           In this setting, we have a short exact sequence $0 \to H \to \HH^0(\tleq F^\bt) \to \HH^0(\tleq C^\bt) \to 0$.
\end{itemize}
\end{Lemma}
\begin{proof}
We have an injective morphism $f : H \to \Cok{d_F^{k-1}}$ since $H$ is a submodule of $\HH^k(F^\bt)$ 
which embedds into $\Cok d_F^{k-1}$. Since $k$ is minimal such that $\HH^k(F^\bt) \neq 0$, we know that
$\HH^j(F^\bt) = 0$ for $j < k$. Hence, we can lift $f$ to a morphism of complexes 
$f^\bt : P^{\leqslant 0}[-k] \to F^\bt$.
In particular, we have $f^0\, d_F^k = 0$ since $f$ factors through $\HH^k(F^\bt)$.
\[
\begin{tikzcd}
\cdots \ar[r] & P^{-2} \ar[r, "d_P^{-2}"] \ar[d, "f^{-2}"] & P^{-1} \ar[r, "d_P^{-1}"] \ar[d, "f^{-1}"] & 
		P^0 \ar[r] \ar[d, "f^0"] & 0 \ar[r] \ar[d] & 0 \ar[r] \ar[d] & \cdots & P^{\leqslant 0}[-k] \ar[d,"f^\bt"] \\
\cdots \ar[r] & F^{k-2} \ar[r, "d_F^{k-2}"] & F^{k-1} \ar[r, "d_F^{k-1}"] & F^k \ar[r, "d_F^k"] & 
		F^{k+1} \ar[r, "d_F^{k+1}"] & F^{k+2} \ar[r] & \cdots & F^\bt
\end{tikzcd}
\] 
We write $C^\bt := C(f)^\bt$ for the mapping cone. 
The distinguished triangle $P^\bt[-k] \to F^\bt \to C^\bt \to$ induces a long exact sequence of cohomology.
\[
\cdots \to \HH^{j-1}(F^\bt) \to \HH^{j-1}(C^\bt) \to \HH^j(P^\bt[-k]) \to \HH^j(F^\bt) \to \HH^j(C^\bt) \to \HH^{j+1}(P^\bt[-k]) \to \cdots
\]
Since $\HH^{<k-1}(F^\bt) = 0$ and $\HH^{<k}(P^\bt[-k]) = 0$, we obtain that $\HH^{<k-1}(C^\bt) = 0$.
Using that $\HH^{k-1}(F^\bt) = 0$ and $\HH^k(P^\bt[-k]) \simeq H$ embedds into $\HH^k(F^\bt)$, we also have $\HH^{k-1}(C^\bt) = 0$.
In particular, we obtain the short exact sequence
$0 \to H \to \HH^k(F^\bt) \to \HH^k(C^\bt) \to 0$ since $\HH^{k+1}(P^\bt[-k]) = 0$. This shows part (1) and (2).

For part (3), suppose that $F^\bt \in \LL_A$ and $k = 0$.
We have that $\HH^{< 0}(C^\bt) = 0$ by part (1). The distinguished triangle $P^\bt \to F^\bt \to C^\bt \to$
from above induces a componentwise split exact sequence $0\to P_\bt^\ast[1] \to C_\bt^\ast \to F_\bt^\ast \to 0$.
Consider the long exact sequence of homology.
\[
\cdots \to \HH_{j}(P_\bt^\ast[1]) \to \HH_{j}(C_\bt^\ast) \to \HH_{j}(F_\bt^\ast) \to \cdots
\] 
Thus, we have $\HH_{\geq 0}(C_\bt^\ast) = 0$ and we obtain $C^\bt = C(f)^\bt \in \LL_A$.
Since $\tau_{\geqslant 0} C^\bt = \tau_{\geqslant 0} F^\bt$, we have 
\[
N:= \Cok\!\big(\HH^0(F^\bt) \hookrightarrow \HH^0(\tleq F^\bt)\big) \simeq \im(d_F^0) \simeq \im(d_C^0) \simeq  \Cok\!\big(\HH^0(C^\bt) \hookrightarrow \HH^0(\tleq C^\bt)\big).
\]
Let $\tilde{H} \simeq \Ker\!\big(\HH^0(\tleq F^\bt) \to \HH^0(\tleq C^\bt)\big)$.
Consider the following commutative diagram with exact rows and columns.
\[ 
\begin{tikzcd}[row sep = 5mm, column sep = 6mm]
         & 0 \ar[d]                & 0 \ar[d]                         & 0 \ar[d]                  &   \\
0 \ar[r] & H \ar[r] \ar[d]         & \HH^0(F^\bt) \ar[r] \ar[d]       & \HH^0(C^\bt) \ar[r] \ar[d]       & 0 \\
0 \ar[r] & \tilde{H} \ar[r] \ar[d] & \HH^0(\tleq F^\bt) \ar[r] \ar[d] & \HH^0(\tleq C^\bt) \ar[r] \ar[d] & 0 \\
         & 0 \ar[r] \ar[d]         & N \ar[r, equal] \ar[d]           & N \ar[r] \ar[d]                  & 0 \\
         & 0                       & 0                                & 0                                & 
\end{tikzcd}
\]
We obtain $H \simeq \tilde{H}$ and the short exact sequence $0 \to H \to \HH^0(\tleq F^\bt) \to \HH^0(\tleq C^\bt) \to 0$.
\end{proof}

\begin{Lemma}\label{Lem:Reduction-To-ProjRes}
Let $\TT$ be a triangulated subcategory of $\KK{A}$ that contains $\LL_A$ and is closed under isomorphisms.
The following holds for a complex $F^\bt \in \Hp A$ with integers $l\leq r$ such that $\HH^{<l}(F^\bt) = 0$ and 
$\HH_{\geqslant r}(F_\bt^\ast) = 0$.
\begin{itemize}
\item[(1)] If $\TT$ contains the minimal projective resolution of $\HH^k(F^\bt)$ for all $k < r$,
		   then $F^\bt \in \TT$.
\item[(2)] We abbreviate $G^\bt := \FF(\HH^r(\tau_{\leq r}F^\bt)) \in \LL_A$. 
           For $l \leq k < r$ let $P^\bt_k$ be the minimal projective resolution of $\HH^k(F^\bt)$.
           We have $[F^\bt] = \sum_{k=l}^{r-1} (-1)^k [P^\bt_k] + (-1)^r[G^\bt]$ in $G_0(\mathcal{\Hp A})$.
\end{itemize}
\end{Lemma}
\begin{proof}
Suppose that $F^\bt \neq 0$. By definition of $\Hp A$, there always exist $l,r \in \ZZ$ 
with $\HH^{<l}(F^\bt) = 0$ and $\HH_{\geqslant r}(F_\bt^\ast) = 0$. 
We can choose $r \in \ZZ$ such that $l\leq r$.
We proceed by induction on $N := r-l$.

If $N = 0$, that is $l = r$, then $\HH^{< r}(F^\bt) = 0$ and $\HH_{\geqslant r}(F_\bt^\ast) = 0$. 
Thus, $F^\bt[r] \in \LL_A$ and we have $F^\bt \in \TT$.
Furthermore, $[F^\bt] = (-1)^r[F^\bt[r]] = (-1)^r[G^\bt]$ since 
$\HH^r(\tau_{\leq r}F^\bt) = \HH^0\!\big(\tau_{\leq 0}(F^\bt[r])\big)$.

We consider the case $N > 0$, that is $l < r$.
By \cref{Lem:RemovingCohomologyOfF}, we have a distinguished triangle
\[
P^\bt[-l] \to F^\bt \to C^\bt \to
\]
with $P^\bt$ the minimal projective resolution of $\HH^l(F^\bt)$.
Moreover, $\HH^j(C^\bt) = 0$ for $j \leq l$ and $\tau_{\geqslant l} C^\bt = \tau_{\geqslant l} F^\bt$.
In particular, this means $\HH^{<l+1}(C^\bt) = 0$ and $\HH_{\geqslant r}(C_\bt^\ast) = 0$.
Hence, by induction, we have $C^\bt \in \TT$ and $[C^\bt] = \sum_{k=l+1}^{r-1} (-1)^k [P^\bt_k] + (-1)^r[G^\bt]$
since $\HH^k(C^\bt) \simeq \HH^k(F^\bt)$ for $k \geq l+1$.
Using that $P^\bt[-l] \rightarrow F^\bt \rightarrow C^\bt \rightarrow$
is a distinguished triangle with $P^\bt \in \TT$, we conclude that $F^\bt \in \TT$ 
since $\TT$ is closed under isomorphisms.
Moreover, we have 
\belowdisplayskip=-10pt
\[
[F^\bt] = (-1)^l[P^{\bt}] + [C^\bt] = \sum_{k=l}^{r-1} (-1)^k [P^\bt_k] + (-1)^r[G^\bt].
\]
\end{proof}

Next, we note the following observation about morphisms ending in injective $A$-modules.
This will be used with $\mathcal{I} = \PP_A$ in this section, 
as well as with $\mathcal{I} = \stp A$ in \cref{Sec:Hstp}.
\begin{Lemma} \label{Lem:X_iff_CompFactor-in-P-perp}
Let $\mathcal{I}$ be a full subcategory of $\inj A$. The following are equivalent for $X \in \rmod A$.
\begin{itemize}
\item[(i)] $X \in {}^\perp \mathcal{I}$.

\item[(ii)] $S \in {}^\perp \mathcal{I}$ for every composition factor $S$ of $X$.
\end{itemize}
\end{Lemma}
\begin{proof}
\textit{Ad (i) $\Rightarrow$ (ii).} Let $S$ be a composition factor of $X$.
Suppose given an injective module $I$ together with a morphism $S \xrightarrow{f} I$.

Since $S$ is a composition factor of $X$, there exists a submodule $M$ of $X$ such that $S \hookrightarrow X/M$.
Using that $I$ is injective, we obtain a morphism $X/M \xrightarrow{g} I$ such that the following diagram commutes.
\[
\begin{tikzcd}
S \ar[r, hookrightarrow] \ar[d, "f"'] & X/M \ar[dl, dashed, "g"] \\
I
\end{tikzcd}
\]
If $f$ is non-zero then the composite map $X \twoheadrightarrow X/M \xrightarrow{g} I$ is non-zero as well.

\textit{Ad (ii) $\Rightarrow$ (i).}
Suppose given an $A$-module $I$ together with a non-zero morphism $f : X \to I$.

Let $S$ be in the socle of $\im f \simeq X/\ker f$. Then $S$ is a composition factor of $X$ and the composite
$S \hookrightarrow \im f \hookrightarrow I$ is non-zero.
\end{proof}

It remains to show that a projective resolution of a module in $\Pperp$ is an element of any
triangulated category that contains $\LL_A$ and is closed under isomorphisms.
Since a short exact sequence of $A$-modules induces a distinguished triangle of projective resolutions in $\KKp A$,
it suffices to consider simple modules.

\begin{Lemma} \label{Lem:ProjRes-if-CompFactors-in-T}
Let $\TT$ be a triangulated subcategory of $\KK{A}$ that is closed under isomorphisms.

Suppose $X$ is an $A$-module with minimal projective resolution $P^{\leqslant 0}$. The following holds.
\begin{itemize}
\item[(1)] If $\TT$ contains the minimal projective resolution of every composition factor $S$ of $X$,
 			 then we have $P^{\leqslant 0} \in \TT$.
 			 
\item[(2)] Let $n:=l(X)$ and suppose $Q^{\leqslant 0}_i$ are the minimal projective resolutions of the 
           composition factors of $X$. Then
           $[P^{\leqslant 0}] = \sum_{i=1}^n [Q^{\leqslant 0}_i]$ in $G_0(\mathcal{\Hp A})$.
\end{itemize}

\end{Lemma}
\begin{proof}
We show the assertions by induction on the length of $X$.
Let $l(X) = 1$. Then $X$ is simple and $P^{\leqslant 0} \in \TT$ by assumption.

Let $l(X) > 1$. Then there exist $A$-modules $S$ and $Y$ with $l(S) = 1$ and $l(Y) < l(X)$ 
such that there is a short exact sequence 
\[
0 \rightarrow S \rightarrow X \rightarrow Y \rightarrow 0.
\]
By assumption and induction respectively, the minimal projective resolutions $Q^{\leqslant 0}$ of $S$ and $R^{\leqslant 0}$ of $Y$ 
are contained in $\TT$. Furthermore, assertion (2) holds for $R^{\leqslant 0}$.
Using the horseshoe lemma, the short exact sequence of modules induces a short exact sequence of projective resolutions.
This, in turn, induces a distinguished triangle in $\KKp A$.
\[
Q^{\leqslant 0} \rightarrow P^{\leqslant 0} \rightarrow R^{\leqslant 0} \rightarrow
\]
Thus, we have $P^{\leqslant 0} \in \TT$ since $\TT$ is closed under isomorphisms.
Moreover, the distinguished triangle implies $[P^{\leqslant 0}] = [Q^{\leqslant 0}] + [R^{\leqslant 0}]$.
\end{proof}

\begin{Lemma} \label{Lem:ProjRes-in-T_Simple}
Suppose $S$ is a simple $A$-module with minimal projective resolution $P^{\leqslant 0}$. 
Let $I$ be the injective hull of $S$ together with a short exact sequence $0 \to S \to I \to C \to 0$ in $\rmod A$.

If $S \in \Pperp_A$, then there exists a distinguished triangle $P^{\leq 0} \to F_I^\bt \to F_C^\bt \to$ in $\KKp A$.
In particular, $P^{\leqslant 0} \in \TT$ for any
triangulated subcategory $\TT$ of $\KKp{A}$ that contains $\LL_A$ and is closed under isomorphisms.
\end{Lemma}
\begin{proof}
Let $I$ be the injective hull of $S$ with embedding $f : S \hookrightarrow I$.
Consider $F_I^\bt \in \LL_A$. By assumption, $I$ is not projective so that $F_I^\bt$ is non-zero.
Our aim is to construct a morphism of complexes $f^\bt : P^{\leqslant 0} \to F_I^\bt$ with $C(f)^\bt \in \LL_A$.
\[
\begin{tikzcd}[row sep=1cm, column sep=1cm]
\cdots \ar[r] & P^{-2} \ar[r, "d_P^{-2}"] \ar[d, "f^{-2}"] & P^{-1} \ar[r, "d_P^{-1}"] \ar[d, "f^{-1}"] &
		P^{0} \ar[r] \ar[d, "f^{0}"] & 0 \ar[r] \ar[d] & \cdots & P^{\leqslant 0} \ar[d,"f^\bt"] \\
\cdots \ar[r] & F_I^{-2} \ar[r, "d_F^{-2}"] & F_I^{-1} \ar[r, "d_F^{-1}"] & 
		F_I^0 \ar[r, "d_{F}^0"] & F_I^1 \ar[r] & \cdots & F_I^\bt 
\end{tikzcd}
\]
The morphism $d_F^0$ factors through $I$ via a morphism $i : I \to F_I^1$.
Since $S$ is simple, the composite map $S \xrightarrow{f} I \xrightarrow{i} F_I^1$ is either injective or zero.
If $fi$ is injective, there exists a morphism $p : F_I^1 \to I$ with $f = (fi)p$ since $I$ is an injective module.
As a consequence, we have $f = f(i\,p) = f(i\,p)^n$ for all $n>0$. Since $S$ is simple,
$I$ is indecomposable and thus the composite $i\,p$ is either an automorphism or $(i\,p)^n$ is zero for some $n>0$.
If $i\,p$ is an automorphism, $i$ is split so that $I$ is projective-injective as a direct
summand of $F_I^1$. However, $I$ is not projective-injective by assumption.
If $(i\,p)^n = 0$, we also have $f = f(i\,p)^n = 0$. A contradiction in both cases.
Thus, the composite $fi$ cannot be injective and must be zero.

This yields 
$S \hookrightarrow \Ker(i) \simeq \Ker\!\big(I \to \im d_F^0\big) \simeq \Ker\!\big(\HH^0(\tleq F_I^\bt) \to F_I^0/\Ker(d_F^0)\big) \simeq \HH^0(F_I^\bt)$.
Since $f$ is non-zero, $S$ is isomorphic to a submodule of $\HH^0(F_I^\bt)$.
Using \cref{Lem:RemovingCohomologyOfF}, we obtain a distinguished triangle 
$P^{\leq 0} \xrightarrow{f^\bt} F_I^\bt \to C(f)^\bt \to$ with $C(f)^\bt \simeq F_C^\bt \in \LL_A$.
\end{proof}

We are now ready to show the main result of this section.
\begin{Theorem} \label{Thm:Triangulated-Hull-of-L}
Let $A$ be a finite dimensional $k$-algebra.

The category $\HP A$ is the smallest triangulated subcategory of $\KKp{A}$ that contains the category $\LL_A$ 
and is closed under isomorphisms.
\end{Theorem} 
\begin{proof}
By \cref{Lem:L-is-in-K_perp}.(3), we have that $\LL_A \subseteq \HP A$. 
By \cref{Rem:Properties_H-HP-Hstp}.(1), $\HP A$ is a triangulated subcategory of $\KKp A$ that is closed under isomorphisms.
Together, we obtain that $\HP A$ is a triangulated category containing $\LL_A$. 

Suppose that $\TT$ is another triangulated subcategory of $\KK{A}$ that contains $\LL_A$ and is closed under isomorphisms.
We show that $\HP A \subseteq \TT$.

Recall that $\HP A \subseteq \KPperp A$. 
By \cref{Lem:H-iff-Complex_in-perp} and \cref{Lem:Reduction-To-ProjRes} it suffices to show that the minimal projective 
resolution of every $A$-module $X \in \Pperp_A$ is an element of $\TT$.
Moreover, by \cref{Lem:ProjRes-if-CompFactors-in-T} it suffices to show that the minimal projective resolution of every 
composition factor of $X$ is an element of $\TT$. Let $S$ be such a composition factor.
Then $S$ is an element of $\Pperp_A$ by \cref{Lem:X_iff_CompFactor-in-P-perp}.
Using \cref{Lem:ProjRes-in-T_Simple}, we now obtain that the minimal projective resolution of $S$ is an element of $\TT$.
\end{proof}

Instead of defining $\HP A$ as a subcategory of $\KPperp A$, we also can consider right perpendicular categories. 
\begin{Remark} \label{Rem:HP-as-RightPerp}
The following are equivalent for a complex $F^\bt \in \KKp A$.
\begin{itemize}
\item[(1)] $F^\bt \in \KPperp A$.
\item[(2)] $H^k(F^\bt) \in \Pperp_A$ for all $k \in \ZZ$.
\item[(3)] $H^k(F^\bt) \in (\nu^{-1}\PP_A)^\perp$ for all $k \in \ZZ$.
\item[(4)] $F^\bt \in \mathcal{K}^b(\nu^{-1}\PP_A)^\perp$.
\end{itemize}
In particular, we have $F^\bt \in \HP A$ if and only if $F^\bt \in \mathcal{K}^b(\nu^{-1}\PP_A)^\perp$ and $F^\bt \in \Hp A$.

In fact, the equivalence of (1) and (2), as well as the equivalence of (3) and (4) 
were shown in \cref{Lem:H-iff-Complex_in-perp}.
The equivalence of (2) and (3) follows from the natural isomorphism
\[
\Hom_A(X,\nu P) \simeq \Du\Hom_A(P,X)
\]
for all $P \in \proj A$ and $X \in \rmod A$.
\end{Remark}

We close this section with a characterization of $\HP A$ inside $\Hp A$ in case that $A$ has dominant dimension at least one.
\begin{Remark} \label{Rem:HomologyInHP_ZeroUnderNu}
Let $\ddim A \geq 1$ and $F^\bt \in \Hp A$.
Then $F^\bt \in \HP A$ if and only if $\nu(\HH^k(F^\bt)) = 0$ for all $k \in \ZZ$.

In fact, we have $F^\bt \in \HP A$ if and only if 
$\HH^k(F^\bt) \in \Pperp_A$ for all $k \in \ZZ$ by \cref{Lem:H-iff-Complex_in-perp}.(1).
However, under the assumption $\ddim A \geq 1$, we have $\Pperp_A = \!{}^\perp\proj A$.
Thus, $\HH^k(F^\bt) \in \Pperp_A$ if and only if $(\HH^k(F^\bt))^\ast = 0$.
\end{Remark}

\section{Grothendieck group}
\label{Sec:Grothendieck}

We recall the definition of the stable Grothendieck group as stated in \cite{MartinezVilla_Properties}.
\begin{Definition} \label{Def:Grothendieck_Stable}
Let $L$ be the free abelian group generated by the isomorphism classes
of objects in $\rmod A$ without projective direct summands.
Let $R$ be the subgroup of $L$ generated by the classes
\[
[X]_{\rm st} - [Y]_{\rm st} + [Z]_{\rm st}
\]
where $0 \to X \oplus P \to Y \oplus Q \to Z \to 0$ is a short exact sequence with 
$P,Q \in \proj A$ and where $X$ and $Y$ may be zero. 

The \textit{stable Grothendieck group $G_0^\mathrm{st}(A)$ of $A$} is defined as the quotient $L/R$.
\end{Definition}
Mart\'inez-Villa has shown in \cite[Theorem 2.1]{MartinezVilla_Properties} that
stably equivalent algebras without nodes and without semisimple summands have isomorphic stable Grothendieck groups.

We consider the Grothendieck group $G_0(\HP A)$ of the triangulated category $\HP A$.
Using the equivalence $\mathcal{F} : \stmod A \to \LL_A$, we obtain a 
Grothendieck group $G_0^\PP(A)$ for the stable module category which is defined via perfect exact sequences.
We show that $G_0^\PP(A)$ is invariant under stable equivalences
which preserve perfect exact sequences.
\begin{Definition} \label{Def:Grothendieck_PerfSeq}
Let $L$ be the free abelian group generated by the isomorphism classes
of objects in $\rmod A$ without projective direct summands.
Let $R'$ be the subgroup of $L$ generated by the classes
\[
[X] - [Y] + [Z]
\]
where $0 \to X \to Y \oplus P \to Z \to 0$ is a perfect exact sequence with $P \in \proj A$. 

The group $G_0^\PP(A)$ is defined as the quotient $L/R'$.
\end{Definition}

\begin{Theorem} \label{Thm:GrothendieckGroup_stmod-HP}
The equivalence $\mathcal{F} : \stmod A \to \LL_A$ induces an isomorphism 
\[
G_0^\PP(A) \simeq G_0(\HP A).
\]
\end{Theorem}
\begin{proof}
By \cref{Prop:PerfSeq-DistTriang}, every perfect exact sequence $0 \to X \to Y \oplus P \to Z \to 0$ in $\rmod A$ with $P$ projective
induces a distinguished triangle 
$F_X^\bt \to F_Y^\bt \to F_Z^\bt \rightarrow $ 
in $\LL_A \subseteq \HP A$ and vice versa.
Hence, the natural map $\sigma : G_0^\PP(A) \to G_0(\HP A)$ given by $[X] \mapsto [F_X^\bt]$ is well-defined.
Since split exact sequences are perfect exact, $\sigma$ is a group homomorphism.

Suppose given $X \in \rmod A$.
Let $S_1,\dots,S_n$ be the composition factors of $X$.
For $1 \leq k \leq n$, let $S_k \hookrightarrow I_k$ be the injective hull of $S_k$ together with a short exact sequence 
$0 \to S_k \to I_k \to C_k \to 0$ in $\rmod A$.
Note that the modules $I_k$ and $C_k$ are uniquely determined by $X$ up to isomorphism.
Throughout the proof, we write $I_X := \bigoplus_{k=1}^n I_k$ and $C_X := \bigoplus_{k=1}^n C_k$ for a given $A$-module $X$.
If $X$ is the zero module, we set $I_X = 0$ and $C_X = 0$.

\textit{Claim 1.} Let $X \in \Pperp_A$ with minimal projective resolution $P^\bt$. 
We have $([I_X] - [C_X])\sigma = [P^\bt]$.

\textit{Proof of claim 1.}
Let $Q_k^\bt$ be the minimal projective resolution of $S_k$ for $1 \leq k \leq n$. 
By \cref{Lem:ProjRes-if-CompFactors-in-T}.(2), we have $[P^\bt] = \sum_{k=1}^n [Q_k]$. 
For all $1 \leq k \leq n$ we additionally have that $[Q_k^\bt] = [F_{I_k}^\bt] - [F_{C_k}^\bt]$ by \cref{Lem:ProjRes-in-T_Simple}.
Together, we obtain 
\[
([I_X] - [C_X])\sigma = \sum_{k=1}^n [I_k]\sigma - [C_k]\sigma 
                      = \sum_{k=1}^n [F_{I_k}^\bt] - [F_{C_k}^\bt] 
                      = \sum_{k=1}^n [Q_k] 
                      = [P^\bt].
\]
This proves the claim.

Suppose given a complex $G^\bt \in \HP A$. 
Let $l \in \ZZ$ such that $\HH^{< l}(G^\bt) = 0$.
Let $r \in \ZZ_{\geq l}$ such that $\HH_{\geq r}(G^\ast_\bt) = 0$.
We aim to define a map 
\[
\tilde{\sigma}' : \HP A \to G_0^\PP(A) : G^\bt \mapsto \sum_{k=l}^{r-1} (-1)^k ([I_{\HH^k(G^\bt)}] - [C_{\HH^k(G^\bt)}]) + (-1)^{r}[\HH^{r}(\tau_{\leq r} G^\bt)].
\]
\textit{Claim 2.}
Suppose given $F_X^\bt \in \LL_A$ for $X\in \rmod A$ with $\HH^0(F_X^\bt)$ non-zero. 
Suppose given a submodule $H$ of $\HH^0(F_X^\bt)$ together with a short exact sequence $0 \to H \to X \to N \to 0$.
We have $[X] = [I_H] - [C_H] + [N]$ in $G_0^\PP(A)$.

If $H = \HH^0(F_X^\bt)$, then we have 
$[X] = [\HH^0(\tleq F_X^\bt)] = [I_H] - [C_H] - [\HH^1(\tau_{\leq 1} F_X^\bt)]$ in $G_0^\PP(A)$.

\textit{Proof of claim 2.} We proceed by induction on the number of composition factors of $H$.

Suppose that $S:= H$ is simple. Let $I \in \rmod A$ be the injective hull of $S$.  
Since $S$ is a submodule of $X$, there exists an injective module $J \in \rmod A$ such that $I \oplus J$ is the injective hull of $X$.
Let $P^\bt$ be the minimal projective resolution of $S$. By \cref{Lem:RemovingCohomologyOfF}.(2,3), we have the following distinguished triangle
\[
P^\bt \xrightarrow{u^\bt} F_X^\bt \to F_N^\bt \to
\]
and a short exact sequence $0 \to S \to \HH^0(F_X^\bt) \to \HH^0(F_N^\bt) \to 0$.
Additionally, we have the following distinguished triangle by \cref{Lem:ProjRes-in-T_Simple} with $w^\bt : P^\bt \to F_I^\bt$ 
induced by the embedding $S \hookrightarrow I$. Note that $S \in \Pperp_A$ since $\HH^0(F_X^\bt) \in \Pperp_A$.
\[
P^\bt \xrightarrow{\matez{w^\bt}{0}} F_I^\bt \oplus F_J^\bt \to F_C^\bt \oplus F_J^\bt \to 
\]
The embedding $S \hookrightarrow I \oplus J$ factors through the embedding $S \hookrightarrow X$ via an injective morphism $X \to I \oplus J$.
This induces a morphism of complexes $({v_1^\bt}\;{v_2^\bt}) : F_X^\bt \to F_I^\bt \oplus F_J^\bt$ such that 
$u^\bt ({v_1^\bt}\;{v_2^\bt}) = ({w^\bt}\;{0})$ in $\KKp A$.
Let $K^\bt := C\!\big(({v_1}\;{v_2})\big)^\bt$ be its mapping cone. 
By \cite[Theorem 3.9]{Kato_Mono}, we have that $K^\bt \in \LL_A$. 
We obtain the following distinguished triangle.
\[
F_X^\bt \xrightarrow{\matez{v_1^\bt}{v_2^\bt}} F_I^\bt \oplus F_J^\bt \to K^\bt \to
\]
Now, the octahedral axiom gives another distinguished triangle.
\[
F_N^\bt \to F_C^\bt \oplus F_J^\bt \to K^\bt \to
\]
By \cref{Prop:PerfSeq-DistTriang} the two triangles above induce perfect exact sequences such that
\[
[X] - [N] = [I] + [J] - [\HH^0(\tleq K^\bt)] - ([C] + [J] - [\HH^0(\tleq K^\bt)]) = [I] - [C].
\]
This verifies the claim in the case that $S = H$ is simple.
Now, suppose that $H$ has $n>1$ many composition factors.
Let $0 \to U \to H \to T \to 0$ be a short exact sequence in $\rmod A$ with $T$ a simple module.
Let $\tilde{X} := \Cok(U \to X)$.
By induction, we may assume that we have $[X] = [I_U] - [C_U] + [\tilde{X}]$ in $G_0^\PP(A)$.
Consider the following commutative diagram with exact rows.
\[
\begin{tikzcd}
0 \ar[r] & U \ar[d, equal] \ar[r] & H \ar[d] \ar[r] & T \ar[d, dashed] \ar[r]         & 0 \\
0 \ar[r] & U \ar[d] \ar[r]        & X \ar[d] \ar[r] & \tilde{X} \ar[d, dashed] \ar[r] & 0 \\
0 \ar[r] & 0 \ar[r]               & N \ar[r, equal] & N \ar[r]                        & 0
\end{tikzcd}
\]
Since $0 \to H \to X \to N \to 0$ is a short exact sequence, 
we obtain another short exact sequence $0 \to T \to \tilde{X} \to N \to 0$.
By \cref{Lem:RemovingCohomologyOfF}.(2,3), we also have a short exact sequence of the form 
$0 \to U \to \HH^0(F_X^\bt) \to \HH^0(F_{\tilde{X}}^\bt) \to 0$. 
Thus, $T$ is a submodule of $\HH^0(F_{\tilde{X}}^\bt)$.
Using that $T$ is simple, we can show as above that $[\tilde{X}] = [I_T] - [C_T] + [N]$.
Since $[I_H] = [I_U] + [I_T]$ and $[C_H] = [C_U] + [C_T]$, combining all equations yields
$[X] = [I_H] - [C_H] + [N]$ in $G_0^\PP(A)$.

If $H = \HH^0(F_X^\bt)$, we have that $\HH^0(F_N^\bt) = 0$ and $\tau_{\geq 0} F_X^\bt = \tau_{\geq 0} F_N^\bt$ 
by \cref{Lem:RemovingCohomologyOfF}.(1,2).
Since $\HH^0(F_N^\bt) = 0$, we know that $F_N^\bt[1] \in \LL_A$.
By \cref{Prop:PerfSeq-DistTriang}, the distinguished triangle $F_N^\bt \to 0 \to F_N^\bt[1] \to$ now results in a
perfect exact sequence $0 \to N \to P \to \HH^0\!\big(\tleq (F_N^\bt[1])\big) \to 0$ for some $P \in \proj A$. 
Thus, we obtain
$[N] = -[\HH^1(\tau_{\leq 1} F_N^\bt)] = -[\HH^1(\tau_{\leq 1} F_X^\bt)]$ in $G_0^\PP(A)$.
In conclusion, $[X] = [I_H] - [C_H] + [N] = [I_H] - [C_H] - [\HH^1(\tau_{\leq 1} F_X^\bt)]$.
This proves the claim.

\textit{Claim 3.} Suppose given a complex $G^\bt \in \HP A$. 
Let $l \in \ZZ$ such that $\HH^{< l}(G^\bt) = 0$.
Let $r_1 \in \ZZ_{\geq l}$ such that $\HH_{\geq r_1}(G^\ast_\bt) = 0$. We write $X_k := \HH^{k}(\tau_{\leq k} G^\bt)$. 
Then the element
\[
  \sum_{k=l}^{r_1-1} (-1)^k ([I_{\HH^k(G^\bt)}] - [C_{\HH^k(G^\bt)}]) + (-1)^{r_1}[X_{r_1}]
\]
in $G_0^\PP(A)$ is independent of the choice of $l$ and $r_1$ provided $\HH^{< l}(G^\bt) = 0$ and $\HH_{\geq r_1}(G^\ast_\bt) = 0$.
That is, for every $r_2 \geq r_1$, we have 
\[
(-1)^{r_1}[X_{r_1}] = \sum_{k=r_1}^{r_2-1} (-1)^k ([I_{\HH^k(G^\bt)}] - [C_{\HH^k(G^\bt)}]) + (-1)^{r_2}[X_{r_2}].
\]
\textit{Proof of claim 3.}
The independence of $l$ of the sum above follows from $I_{\HH^k(G^\bt)} = 0$ and $C_{\HH^k(G^\bt)} = 0$ for $k < l$.
We show that 
\[
0 = \sum_{k=r_1}^{r_2-1} (-1)^k ([I_{\HH^k(G^\bt)}] - [C_{\HH^k(G^\bt)}]) + (-1)^{r_2}[X_{r_2}] - (-1)^{r_1}[X_{r_1}].
\]
by induction on $r_2 \in \ZZ_{\geq r_1}$. 
For $r_2 = r_1$ there is nothing to show.
For the induction step, we may assume that the equation holds for some $r_2 \geq r_1$. We obtain
\begin{align*}
 & \sum_{k=r_1}^{r_2} (-1)^k ([I_{\HH^k(G^\bt)}] - [C_{\HH^k(G^\bt)}]) + (-1)^{r_2+1}[X_{r_2+1}] - (-1)^{r_1}[X_{r_1}] \\
=&\; (-1)^{r_2} ([I_{\HH^{r_2}(G^\bt)}] - [C_{\HH^{r_2}(G^\bt)}]) + (-1)^{r_2+1}[X_{r_2+1}] - (-1)^{r_2}[X_{r_2}].
\end{align*}
Note that $\HH_{\geq r_2}(G^\ast_\bt) = 0$ and $\HH_{\geq r_2}((F_{X_{r_2}}^\bt[-r_2])^\ast) = \HH_{\geq 0}((F_{X_{r_2}}^\bt)^\ast) = 0$.
Furthermore, we have $\HH^{r_2}(\tau_{\leq r_2}(F_{X_{r_2}}^\bt[-r_2])) = \HH^{0}(\tleq F_{X_{r_2}}^\bt) \simeq X_{r_2} = \HH^{r_2}(\tau_{\leq r_2} G^\bt)$. 
Since projective resolutions are unique up to isomorphism in $\KKp A$, this implies that 
$\tau_{\geq r_2}G^\bt \simeq \tau_{\geq r_2}(F_{X_{r_2}}^\bt[-r_2])$ in $\KKp A$.
In particular, we have $\HH^{r_2}(G^\bt) \simeq \HH^{r_2}(F_{X_{r_2}}^\bt[-r_2]) = \HH^{0}(F_{X_{r_2}}^\bt)$. 
Using claim 2, we obtain
\begin{align*}
[I_{\HH^{r_2}(G^\bt)}] - [C_{\HH^{r_2}(G^\bt)}] 
    &= [\HH^{0}(\tleq F_{X_{r_2}}^\bt)] + [\HH^{1}(\tau_{\leq 1} F_{X_{r_2}}^\bt)] \\
    &= [\HH^{r_2}(\tau_{\leq r_2}G^\bt)] + [\HH^{r_2+1}(\tau_{\leq r_2+1}G^\bt)] \\
    &= [X_{r_2}] + [X_{r_2+1}].
\end{align*}
Thus, we have
\begin{align*}
 \; (-1)^{r_2} ([I_{\HH^{r_2}(G^\bt)}] - [C_{\HH^{r_2}(G^\bt)}]) + (-1)^{r_2+1}[X_{r_2+1}] - (-1)^{r_2}[X_{r_2}] &
= \\ (-1)^{r_2} ([X_{r_2}] + [X_{r_2+1}]) + (-1)^{r_2+1}[X_{r_2+1}] - (-1)^{r_2}[X_{r_2}] &
=\; 0.
\end{align*}
This proves the claim.

Suppose given a complex $G^\bt \in \HP A$. 
Let $l \in \ZZ$ such that $\HH^{< l}(G^\bt) = 0$.
Let $r \in \ZZ_{\geq l}$ such that $\HH_{\geq r}(G^\ast_\bt) = 0$.
We define a map $\tilde{\sigma}'$ as follows.
\[
\tilde{\sigma}' : \HP A \to G_0^\PP(A) : G^\bt \mapsto \sum_{k=l}^{r-1} (-1)^k ([I_{\HH^k(G^\bt)}] - [C_{\HH^k(G^\bt)}]) + (-1)^{r}[\HH^{r}(\tau_{\leq r} G^\bt)].
\]
By claim 3 this definition is independent of the choice of $l$ and $r$.

\textit{Claim 4.}
Suppose given a distinguished triangle $G^\bt \to K^\bt \to L^\bt \to$ in $\HP A$. We have that
$([G^\bt]-[K^\bt]+[L^\bt])\tilde{\sigma}' = 0$ in $G_0^\PP(A)$.

\textit{Proof of claim 4.}
Let $l \in \ZZ$ such that $\HH^{< l}(G^\bt) = 0$, $\HH^{< l}(K^\bt) = 0$ and $\HH^{< l}(L^\bt) = 0$.
Let $r \in \ZZ_{\geq l}$ such that $\HH_{\geq r}(G_\bt^\ast) = 0$, $\HH_{\geq r}(K_\bt^\ast) = 0$ and $\HH_{\geq r}(L_\bt^\ast) = 0$.
The distinguished triangle induces a split short exact sequence of complexes
$0 \to \tau_{\leq r}G^\bt \to \tau_{\leq r}\tilde{K}^\bt \to \tau_{\leq r}L^\bt \to 0$
with $\tilde{K}^\bt \simeq K^\bt$ in $\KKp A$. 
Note that $[\HH^{r}(\tau_{\leq r} \tilde{K}^\bt)] = [\HH^{r}(\tau_{\leq r} K^\bt)]$.
This results in the following long exact sequence of cohomology.
\[
0 \to \HH^l(G^\bt) \to \HH^l(K^\bt) \to \cdots \to \HH^{r-1}(L^\bt) \xrightarrow{\delta} \HH^r(\tau_{\leq r}G^\bt) \to \HH^r(\tau_{\leq r}\tilde{K}^\bt) \to \HH^r(\tau_{\leq r}L^\bt) \to 0
\]
Recall that for $X \in \rmod A$ the modules $I_X$ and $C_X$ are uniquely determined up to isomorphism 
by the composition factors of $X$. The long exact sequence of cohomology now implies that 
\small
\[
\sum_{k=l}^{r-1} (-1)^k ([I_{\HH^k(G^\bt)}] - [I_{\HH^k(K^\bt)}] + [I_{\HH^k(L^\bt)}] - [C_{\HH^k(G^\bt)}] + [C_{\HH^k(K^\bt)}] - [C_{\HH^k(L^\bt)}]) + (-1)^r([I_{\im(\delta)}] - [C_{\im(\delta)}])
=0
\]\normalsize
We write $X:= \HH^r(\tau_{\leq r}G^\bt)$ and $\tilde{X} := \HH^r(\tau_{\leq r}G^\bt)/\im(\delta) \simeq \Ker\!\big(\HH^r(\tau_{\leq r}\tilde{K}^\bt) \to \HH^r(\tau_{\leq r}L^\bt)\big)$.
Note that $\im(\delta) \simeq \Cok\!\big(\HH^{r-1}(K^\bt) \to \HH^{r-1}(L^\bt)\big) \hookrightarrow \HH^r(G^\bt) \simeq \HH^0(F_X^\bt)$.
Thus, \cref{Lem:RemovingCohomologyOfF}.(1,3) implies $\tau_{\geq 0}F_X^\bt = \tau_{\geq 0}F_{\tilde X}^\bt$ so that $X^\ast \simeq \tilde{X}^\ast$.
Moreover, $[\HH^r(\tau_{\leq r}G^\bt)] = [X] = [I_{\im(\delta)}] - [C_{\im(\delta)}] + [\tilde{X}]$ by claim 2.
Together with the above, we obtain
\begin{align*}
     & \; ([G^\bt]-[K^\bt]+[L^\bt])\tilde{\sigma}'  \\
    =& \; \sum_{k=l}^{r-1} (-1)^k ([I_{\HH^k(G^\bt)}] - [I_{\HH^k(K^\bt)}] + [I_{\HH^k(L^\bt)}] - [C_{\HH^k(G^\bt)}] + [C_{\HH^k(K^\bt)}] - [C_{\HH^k(L^\bt)}]) \\
     & \; + (-1)^{r}([\HH^{r}(\tau_{\leq r} G^\bt)] - [\HH^{r}(\tau_{\leq r} K^\bt)] + [\HH^{r}(\tau_{\leq r} L^\bt)]) \\
    =& \; -(-1)^r([I_{\im(\delta)}] - [C_{\im(\delta)}])  + (-1)^{r}([\HH^{r}(\tau_{\leq r} G^\bt)] - [\HH^{r}(\tau_{\leq r} K^\bt)] + [\HH^{r}(\tau_{\leq r} L^\bt)]) \\
    =& \; (-1)^r([\tilde{X}] - [X] + [\HH^{r}(\tau_{\leq r} G^\bt)] - [\HH^{r}(\tau_{\leq r} K^\bt)] + [\HH^{r}(\tau_{\leq r} L^\bt)]) \\
    =& \; (-1)^r([\tilde{X}] - [\HH^{r}(\tau_{\leq r} \tilde{K}^\bt)] + [\HH^{r}(\tau_{\leq r} L^\bt)])
\end{align*}
It remains to show that the short exact sequence 
$0 \to \tilde{X} \to \HH^r(\tau_{\leq r}\tilde{K}^\bt) \to \HH^r(\tau_{\leq r}L^\bt) \to 0$
is perfect exact. The componentwise split sequence 
$0 \to \tau_{\geq r+1}L_\bt^\ast \to \tau_{\geq r+1}\tilde{K}_\bt^\ast \to \tau_{\geq r+1}G_\bt^\ast \to 0$
induces the following short exact sequence.
\[
0 \to \HH_{r+1}(\tau_{\geq r+1} L^\ast_\bt) \to \HH_{r+1}(\tau_{\geq r+1} \tilde{K}^\ast_\bt) \to \HH_{r+1}(\tau_{\geq r+1} G^\ast_\bt) \to 0
\]
Recall that $\HH_{\geq r}(L_\bt^\ast) = 0$. Thus, we have 
\[
\HH_{r+1}(\tau_{\geq r+1} L^\ast_\bt) \simeq \Cok(L_{r+2}^\ast \to L_{r+1}^\ast) \simeq \Ker(L_r^\ast \to L_{r-1}^\ast) 
\simeq \HH^r(\tau_{\leq r}L^\bt)^\ast.
\]
Similarly for the other two terms. This results in the following short exact sequence.
\[
0 \to \HH^r(\tau_{\leq r}L^\bt)^\ast \to \HH^r(\tau_{\leq r}\tilde{K}^\bt)^\ast \to X^\ast \to 0
\]
Since $X^\ast \simeq \tilde{X}^\ast$, we obtain that $0 \to \tilde{X} \to \HH^r(\tau_{\leq r}\tilde{K}^\bt) \to \HH^r(\tau_{\leq r}L^\bt) \to 0$
is a perfect exact sequence. This proves the claim.

Using claim 4, the map $\tilde{\sigma}'$ induces a map
\[
\tilde{\sigma} : G_0(\HP A) \to G_0^\PP(A) : [G^\bt] \mapsto \sum_{k=l}^{r-1} (-1)^k ([I_{\HH^k(G^\bt)}] - [C_{\HH^k(G^\bt)}]) + (-1)^{r}[\HH^{r}(\tau_{\leq r} G^\bt)].
\]
We show that $\sigma\, \tilde{\sigma} = \id_{G_0^\PP(A)}$ and $\tilde{\sigma}\, \sigma = \id_{G_0(\HP A)}$.
Then $\sigma$ and $\tilde{\sigma}$ are mutually inverse isomorphisms.

For $X \in \rmod A$, we have $[X] \sigma\, \tilde{\sigma} = [F_X^\bt] \tilde{\sigma} = [\HH^0(\tleq F_X^\bt)] = [X]$
since we can choose $l = r = 0$ in the definition of $\tilde{\sigma}$.
On the other hand, suppose given $G^\bt \in \HP A$. 
Let $P_k^\bt$ be the minimal projective resolution of $\HH^k(G^\bt)$ for $l \leq k \leq r-1$ where $l,r \in \ZZ$ 
with $\HH^{<l}(G^\bt)=0$ and $\HH_{\geq r}(G_\bt^\ast) = 0$. Let $X_r := \HH^r(\tau_{\leq r}G^\bt)$.
Using claim 1, we have $([I_{\HH^k(G^\bt)}] - [C_{\HH^k(G^\bt)}])\sigma = [P_k^\bt]$.
By \cref{Lem:Reduction-To-ProjRes}.(2), we have 
$[G^\bt] = \sum_{k=l}^{r-1} (-1)^k [P_k^\bt] + (-1)^{r}[F_{X_{r}}^\bt]$.
Together, we obtain
\belowdisplayskip=-13pt
\begin{align*}
[G^\bt]\tilde{\sigma}\, \sigma &= \sum_{k=l}^{r-1} (-1)^k ([I_{\HH^k(G^\bt)}] - [C_{\HH^k(G^\bt)}])\sigma + (-1)^{r}[X_r]\sigma \\
                               &= \sum_{k=l}^{r-1} (-1)^k [P_k^\bt] + (-1)^{r}[F_{X_{r}}^\bt]
                                = [G^\bt].
\end{align*}
\end{proof}

Let $\eta : 0 \rightarrow X \to Y \oplus P \to Z \rightarrow 0$ be a perfect exact sequence in $\rmod A$
without split summands such that $P$ is projective and $Y$ has no projective direct summand.
We say that a stable equivalence $\alpha : \stmod A \to \stmod B$ preserves the perfect exact sequence $\eta$
if there exists a perfect exact sequence
$0 \rightarrow \alpha(X) \to \alpha(Y) \oplus \tilde{P} \to \alpha(Z) \rightarrow 0$
in $\rmod B$ with $\tilde{P} \in \proj B$.
Similar to the stable Grothendieck group, a stable equivalence that preserves perfect exact sequences 
also induces an isomorphism on the level of $G_0^\PP(A)$.
A stable equivalence $\stmod A \to \stmod B$ preserves perfect exact sequences at least in a situation where
$A$ and $B$ are of finite representation type and have no nodes; cf.\ \cite[Corollary 3.18]{Nitsche_StableEquivalences}.
\begin{Theorem} \label{Thm:InducedIsoOnGrothendieckGroup}
Let $\alpha : \stmod A \to \stmod B$ be a stable equivalence such that $\alpha$ and its
quasi-inverse preserve perfect exact sequences.
Then $\alpha$ induces an isomorphism $G_0^\PP(A) \xrightarrow{\sim} G_0^\PP(B)$.
\end{Theorem}
\begin{proof}
By assumption, every perfect exact sequence $0 \to X \to Y \oplus P \to Z \to 0$ in $\rmod A$ 
with $P$ projective induces a perfect exact sequence
\[
0 \rightarrow \alpha(X) \rightarrow \alpha(Y) \oplus \tilde{P} \rightarrow \alpha(Z) \rightarrow 0
\]
with $\tilde{P} \in \proj B$.

Hence, the natural map $G_0^\PP(A) \rightarrow G_0^\PP(B)$ given by $[X] \mapsto [\alpha(X)]$ is well-defined.
Since $\alpha$ is an equivalence, this map is an isomorphism.
\end{proof}

Note that $G_0^\mathrm{st}(A) = 0$ if $\gdim A < \infty$. In fact, we have $[\Omega(X)] = -[X]$ by setting $Y = 0$
in the definition of $G_0^\mathrm{st}(A) = 0$. If $\Omega^n(X)$ is projective for some $n \geq 1$, we obtain $[X] = 0$.

In general, $G_0^\mathrm{st}(A)$ and $G_0^\PP(A)$ are not isomorphic.
In particular, $G_0^\PP(A)$ can be non-zero even in case of finite global dimension.
However, the following holds.
\begin{Remark} \label{Rem:GrothendieckGroup_st-P}
We have a surjective group homomorphism
\[
G_0^\PP(A) \to G_0^\mathrm{st}(A) : [X] \mapsto [X]_\mathrm{st}\,.
\]
If $A$ is self-injective, this is an isomorphism.

In fact, every perfect exact sequence in the definition of $G_0^\PP(A)$ is also a short exact sequence
of the form stated in the definition of $G_0^\mathrm{st}(A)$.
If $A$ is self-injective, every short exact sequence is perfect exact.
After potentially removing a split exact sequence starting in a projective module, every short exact sequence
is of the form as stated in \cref{Def:Grothendieck_PerfSeq}.
\end{Remark}

We close this section with a remark on generating systems of $G_0^\PP(A)$.
\begin{Remark} \label{Rem:GrothendieckGroup-GeneratingSystem}
By construction, $G_0^\PP(A)$ is generated by the indecomposable modules in $\stmod A$.
Thus, $G_0(\HP A)$ is generated by $[F^\bt]$ for $F^\bt \in \LL_A$ indecomposable by \cref{Thm:GrothendieckGroup_stmod-HP}.

However, in general $G_0^\PP(A)$ is not generated by the non-projective simple modules in $\stmod A$.
In comparison, every simple minded system over $A$ is a generating system of $G_0^\mathrm{st}(A)$ as was shown in 
\cite[Lemma 2.3]{KoenigLiu_sms}. In particular, this holds for the simple $A$-modules.
\end{Remark}

%
\section{A Nakayama closure}
\label{Sec:Hstp}

In general, the triangulated category $\HP A$ discussed in \cref{Sec:Category_HP(proj A)} is neither 
characteristic in $\KKp A$, nor in $\Hp A$.  In particular, for algebras of finite global dimension, the
category is not closed under the derived Nakayama functor.
However, the category $\HP A$ can be enlarged 
to the category $\Hstp A$ by replacing the projective-injective modules with strongly projective-injective modules.

In this section, we consider an equivalence $\nu_{\mathcal{K}} : \KKpbsp A \to \KKmbp A$ induced by the
Nakayama functor. 
In case that $\gdim A < \infty$, we retrieve the derived Nakayama functor $\KKbp A \to \KKbp A$. 
Our aim is to show that $\Hstp A$ is the smallest triangulated 
subcategory of $\KKp A$ that contains $\LL_A$ and is closed under $\nu_{\mathcal{K}}$ and under isomorphisms.
Assuming that $A$ can be embedded into a strongly projective-injective module, we give conditions under which
$\Hstp A$ is characteristic in $\Hp A$.

For the main proof, we will be able to reuse most of the results of \cref{Sec:Category_HP(proj A)}.
The main new technical result in \cref{Lem:ProjRes-in-T_Simple_nu} shows that a projective resolution 
of a simple module in $\stperp A$ is contained in any triangulated subcategory 
that contains $\LL_A$ and is closed under $\nu_{\mathcal{K}}$ and under isomorphisms.

\begin{Lemma} \label{Lem:nu-fully-faithful_special-case}
Let $X$ be an $A$-module and $Z \in \PP_A$. Then 
$\Hom_A(\nu^{-1} X, Z) \simeq \Hom_A(X, \nu Z)$ as $k$-vector spaces.
\end{Lemma}
\begin{proof}
Let $I^\bt \in \KKpi{A}$ be an injective presentation of $X$.
\[
\begin{tikzcd}
 X \ar[r, hookrightarrow] & I^0 \ar[r, "d_I"] & I^1
\end{tikzcd}
\]
Applying $\nu^{-1}$ componentwise, we obtain a sequence $Q^\bt \in \KKpp A$.
\[
\begin{tikzcd}
\nu^{-1} X \ar[r, hookrightarrow, "\iota"] & Q^0 \ar[r, "d_Q"] & Q^1
\end{tikzcd}
\]
Note that we have $X = \ker(d_I)$ and $\nu^{-1}X = \ker(d_Q)$ since $\nu^{-1}$ is left exact.

\textit{Claim.} We have $ \Hom_A(X, Z) \simeq \Hom_{\KKi{A}}(I^\bt, Z)$ and 
                $ \Hom_A(\nu^{-1} X, Z) \simeq \Hom_{\KKp{A}}(Q^\bt, Z)$ for $Z \in \inj A$.

Suppose given $Y \in \rmod A$ and a sequence $0 \to Y \to C^0 \to C^1$ in $\rmod A$ with $\Ker(C^0 \xrightarrow{d} C^1)$.
This gives a complex $C^\bt \in \KK A$ with $C^{k} = 0$ for $k \not\in \{0,\,1\}$.
We show that this implies $\Hom_A(Y, Z) \simeq \Hom_{\KK{A}}(C^\bt, Z)$.

Let $f$ be a non-zero morphism in $\Hom_A(Y, Z)$.
Since $Z$ is injective, there is a morphism $\varphi^0 : C^0 \to Z$ such that the following diagram commutes.
\[
\begin{tikzcd}
Y \ar[r, hookrightarrow, "\iota"] \ar[d, "f"] & C^0 \ar[dl, "\varphi^0" close] \ar[r, "d"] & C^1 \\
Z
\end{tikzcd}
\]
Set $\varphi^k = 0$ for all $k \neq 0$. We obtain $\varphi^\bt \in \Hom_{\KK{A}}(C^\bt, Z)$.

Let $C^0 \xrightarrow{\tilde{\varphi}^0} Z$ be another morphism such that $f = \iota \, \tilde{\varphi}^0$. 
Then $\iota (\varphi^0 - \tilde{\varphi}^0) = 0$ and we obtain the following morphism of complexes.
\[
\begin{tikzcd}
Y \ar[r, "\iota"] \ar[d] & C^0 \ar[r, "d"] \ar[d, "\varphi^0 - \tilde{\varphi}^0"] & C^1 \ar[d] \\
0 \ar[r] & Z \ar[r] & 0 
\end{tikzcd}
\] 
By \cref{Lem:Zero-Hom_At-exact-Pos}$.(1)$ this yields $\varphi^0 - \tilde{\varphi}^0 = 0$ so that $f\varphi^0 = \tilde{\varphi}^0$.
Thus, $f \mapsto f\psi := \varphi^\bt$ defines a $k$-linear map 
\[
\Hom_A(Y, Z) \xrightarrow{\psi} \Hom_{\KK{A}}(C^\bt, Z).
\]
It remains to show that $\psi$ is an isomorphism.
Suppose that $\varphi^\bt = f\psi = 0$ in $\KK{A}$. In this case, there exists a morphism 
$h : C^1 \to Z$ such that $d\, h = \varphi^0$.
However, this implies 
\[
0 = \iota \, d \, h = \iota \, \varphi^0 = f\]
so that $\psi$ is injective.
Now, suppose given a morphism $\varphi^\bt \in \Hom_{\KK{A}}(C^\bt, Z)$. Setting
$f := \iota\,\varphi^0 \in \Hom_A(Y, Z)$,
we obtain $f\psi = \varphi^\bt$ so that $\psi$ is surjective.
This concludes the proof of the claim.

We obtain the following sequence of isomorphisms using that $\nu : \KKp A \to \KKi A$ is an equivalence.
\belowdisplayskip=-13pt
\begin{align*}
\Hom_A(\nu^{-1} X, Z) &\simeq \Hom_{\KKp{A}}(Q^\bt, Z)  \\
	&\simeq \Hom_{\KKi{A}}(I^\bt, \nu Z)                \\
	&\simeq \Hom_A(X, \nu Z)
\end{align*}
\end{proof}

The next lemma provides one part of the functor $\nu_{\mathcal{K}} : \KKpbsp A \to \KKmbp A$. 
Recall that the equivalence $\nu : \proj A \to \inj A$ given by the Nakayama functor 
$\nu(-) = D\big((-)^\ast\big)$ induces an equivalence $\KKp{A} \to \KKi A$ via $F^\bt \mapsto (\nu F^k)_{k \in \ZZ}$.
\begin{Lemma} \label{Lem:nu-equivalence_restricted}
The equivalence of triangulated categories
\[
\nu : \KKp{A} \xrightarrow{\sim} \KKi A
\]
induced by the Nakayama functor restricts to an equivalence of triangulated categories
\[
\KKpbsp A \xrightarrow{\sim} \KKpbi A.
\]
\end{Lemma}
\begin{proof}
Recall that $\nu(X) = \Du(X^\ast)$ for an $A$-module $X$.
Suppose given $P^\bt \in \KKpbsp A $. Thus, we have $P_\bt^\ast \in \mathrm{K}^{-,b}(A\text{-proj})$
and we obtain $\Du(P_\bt^\ast) = \nu P^\bt \in \KKpbi A$ since $\Du$ is exact.

Recall that $\nu^{-1}(X) = \big(\Du(X)\big)^\ast$ for an $A$-module $X$.
Suppose given $I^\bt \in \KKpbi A$.
Applying $\Du(-)$ to $I^\bt$ componentwise, we obtain a complex
$\Du I^\bt \in \mathrm{K}^{-,b}(A\text{-proj})$ since $\Du$ is exact.
In particular $\Du I^\bt$ is bounded in cohomology so that $\big(\Du I^\bt\big)^\ast = \nu^{-1} I^\bt \in \KKpbsp A$.

In conclusion, $\nu$ restricts to an equivalence $\KKpbsp A \xrightarrow{\sim} \KKpbi A$.
\end{proof}

Recall that the canonical functors $\KKpl A \to \DDpl A$ and $\KKm A \to \DDm A$ 
induce equivalences of triangulated categories.
\[
\KKpbi A \xrightarrow{\sim} \mathrm{D}^b(\rmod A) \xrightarrow{\sim} \KKmbp A.
\] 
Composing these equivalences with the one from \cref{Lem:nu-equivalence_restricted} 
yields an equivalence of triangulated categories $\KKpbsp A \to \KKpbi A \to \KKmbp A$.
\begin{Definition} \label{Def:nu_K}
We denote the above composite of equivalences by
\[
\nu_{\mathcal{K}} \,\colon\, \KKpbsp A \xrightarrow{\sim} \KKmbp A.			
\]
We say that a triangulated subcategory $\TT$ of $\KK{A}$ is \textit{closed under} $\nu_{\mathcal{K}}$ if 	
the restriction of $\nu_{\mathcal{K}}$ to the full subcategory with objects in 
$\TT \cap \KKpbsp A$ has an essential image in $\TT$.
\end{Definition}

\begin{Remark}
Recall that $\Hp A$ is the full subcategory of $\KKp A$ consisting of all complexes 
$F^\bt \in \KKp A$ such that there exist $l,r \in \ZZ$ with $\HH^{< l}(F^\bt) = 0$ and 
$\HH_{\geqslant r}(F_\bt^\ast) = 0$. In particular, $\KKpbsp A$ and $\KKmbp A$ are subcategories of $\Hp A$.
\begin{itemize}
\item[(1)] Let $\gdim A < \infty$. Then $\Hstp A = \Kstperp A \cap \KKbp A$ and
		$\nu_{\mathcal{K}}$ is equivalent to the derived Nakayama functor $\KKbp{A} \xrightarrow{\sim} \KKbp{A}$.

\item[(2)] If $A$ is self-injective, we will see that $\Hstp A = \HP A = \LL_A$; cf.\ \cref{Thm:Charact_SelfInjective}.
		   Furthermore, the restriction of $\nu_{\mathcal{K}} $ to $\Hstp A \cap \KKpbsp A$ is zero, since
		   all non-zero complexes in $\Hstp A$ are unbounded. Thus, $\LL_A$ is trivially closed under $\nu_{\mathcal{K}}$
		   in this case.
\end{itemize}
\end{Remark}

We aim to show that $\Hstp A$ is the smallest triangulated 
subcategory of $\KKp A$ that contains $\LL_A$ and is closed under $\nu_{\mathcal{K}}$ and under isomorphisms.
It remains to verify that a projective resolution $P^{\leqslant 0}$ of a simple module 
$S \in \stperp A$ is an element of every triangulated category 
that contains $\LL_A$ and is closed under $\nu_{\mathcal{K}}$ and under isomorphisms. 
Our strategy is as follows.

We consider a complex $Q^\bt \in \KKp A$ such that $\nu_{\mathcal{K}} (Q^{> 0}) \simeq P^{\leqslant 0}$ and $Q^{\leqslant 0}$
is a projective resolution of $\nu^{-1}S$. Furthermore, $\nu^{-1}S$ embeds into $Q^1$.
If $Q^1$ is not injective, the result follows from \cref{Lem:ProjRes-in-T_Simple}.
Otherwise, we inductively construct new simple modules $S_k$ which embed into $\nu^{-k}(Q^1)$.
By assumption, $\nu(Q^1)$ cannot be strongly projective-injective, so that this procedure terminates with a $\nu^{-k}(Q^1)$ 
which is not injective.
\begin{Lemma} \label{Lem:ProjRes-in-T_Simple_nu}
Let $\TT$ be a triangulated subcategory of $\KKp{A}$ that contains $\LL_A$ and is closed under $\nu_{\mathcal{K}}$
and under isomorphisms.

Suppose $S$ is a simple $A$-module with minimal projective resolution $P^{\leqslant 0}$. 
If $S \in \stperp A$, then we have $P^{\leqslant 0} \in \TT$.
\end{Lemma}
\begin{proof}
Suppose that $S$ is injective. Then we have $S^\ast = 0$, otherwise $S$ is projective and thus 
strongly projective-injective, a contradiction.
Hence, $F_S^\bt$ is the minimal projective resolution of $S$ which is contained in $\LL_A$.

For the remainder of the proof we assume that $S$ is not injective.
Let $S \hookrightarrow \nu Q^1 \rightarrow \nu Q^2$ be the minimal injective presentation of $S$ with $Q^1, Q^2 \in \proj A$.

Extend $Q^1 \rightarrow Q^2$ to an element $Q^\bt[2] \in \LL_A$. 
Then the truncation $Q^{\leqslant 0}$ is the minimal projective resolution of $\nu^{-1} S = \ker{\left( Q^1 \rightarrow Q^2 \right)}$.
\[
\begin{tikzcd}
\cdots \ar[r] & Q^{-1} \ar[r] & Q^0 \ar[rr] \ar[dr, twoheadrightarrow] & & Q^1 \ar[r] & Q^2 \ar[r] & \cdots \\
 & & & \nu^{-1} S \ar[ur, hookrightarrow] &
\end{tikzcd}
\]
Note that $Q^{> 0} \in \KKpbsp A$.
Applying $\nu_{\mathcal{K}}$ yields $\nu_{\mathcal{K}} (Q^{> 0}) \simeq P^{\leqslant 0}$ in $\KKmbp A$ since $\nu Q^{> 0}$ is an injective resolution of $S$.
Hence, $Q^{> 0} \in \TT$ implies $P^{\leqslant 0} \in \TT$ since $\TT$ is closed under $\nu_{\mathcal{K}}$.
Moreover, we have the distinguished triangle 
\[
Q^{> 0} \rightarrow Q^\bt \rightarrow Q^{\leqslant 0} \rightarrow Q^{> 0}[1]
\]
so that $Q^{> 0} \in \TT$ if and only if $Q^{\leqslant 0} \in \TT$ since $\TT$ is closed under isomorphisms.
It remains to show that $Q^{\leqslant 0} \in \TT$.
If $\nu^{-1}S = 0$, we have $Q^{\leqslant 0} = 0$ and we are done at this point.

Note that $\dim_k \Hom_A(S, \nu Q^1) = 1$ and $\dim_k \Hom_A(S, I) = 0 $ for any injective module $I \not\simeq \nu Q^1$ 
since $S$ is simple. Suppose that $\nu Q^1$ is projective-injective. 
Otherwise $S \in {}^\perp \mathcal{P}$ and thus $P^{\leqslant 0} \in \TT$ by \cref{Lem:ProjRes-in-T_Simple}.
By assumption, $\nu Q^1$ is not strongly projective-injective.

\textit{Claim.} It suffices to consider the case that $Q^1 \in \proj A$ is not injective.

Assume that $Q^1 \in \PP_A$. For every $Z \in \PP_A$ with $Z \not\simeq Q^1$, 
that is $\nu Z \not\simeq \nu Q^1$, we have
\begin{align*}
\dim_k \Hom_A(\nu^{-1} S, Q^1) &= \dim_k \Hom_A(S, \nu Q^1) = 1 \\
\dim_k \Hom_A(\nu^{-1} S, Z) &= \dim_k \Hom_A(S, \nu Z) = 0
\end{align*}
by \cref{Lem:nu-fully-faithful_special-case}. 
Therefore, there exists a unique composition factor $S_0$ of $\nu^{-1} S$ with an embedding into $Q^1$.
Furthermore, every other composition factor $S'$ of $\nu^{-1} S$ lies in $\Pperp_A$.

By \cref{Lem:ProjRes-in-T_Simple} this means that the minimal projective resolution of every composition factor 
which is not isomorphic to $S_0$ is an element of $\TT$.
Therefore, by \cref{Lem:ProjRes-if-CompFactors-in-T}, the minimal projective resolution 
$Q^{\leqslant 0}$ of $\nu^{-1} S$ is an element of $\TT$
if the minimal projective resolution of $S_0$ is an element of $\TT$.
Note that $S_0 \in {}^\perp \stp A$. If not, then $Q^1$ must be strongly projective-injective 
and therefore $\nu Q^1$ as well so that $S \not\in \stperp A$. 

Now we can repeat the process described above for $S_0$ instead of $S$ and $Q^1$ instead of $\nu Q^1$.
Inductively, for $k \geq 0$, this results in a simple module $S_k \in \stperp A$ which is a composition 
factor of $\nu^{-1}(S_{k-1})$. Furthermore, $S_k$ embeds into $\nu^{-k}(Q^1)$.
Since $A$ is finite dimensional and $Q^1$ is not strongly projective-injective, 
there exists a $k \in \ZZ$ such that $\nu^{-k}(Q^1)$ is projective but not injective.
Thus, it suffices to show that the minimal projective resolution of $S_k$ is in $\mathcal T$.
This concludes the proof of the claim.

As a result, we can assume that $Q^1$ is not injective.
We show that $\nu^{-1} S \in \Pperp_A$.
If not, there is a $Z \in \PP_A$ such that $\dim_k \Hom(\nu^{-1} S, Z) \neq 0$.
However, using that $\nu Q_1$ is the injective hull of $S$, we have
\[
\dim_k \Hom(\nu^{-1} S, Z) \simeq \dim_k \Hom(S, \nu Z) = 0
\]
since $Z \not\simeq Q^1 \not\in \PP_A$, that is $\nu Z \not\simeq \nu Q^1$.
This yields $\nu^{-1} S \in \Pperp_A$.
Hence, the minimal projective resolution $Q^{\leqslant 0}$ of $\nu^{-1} S$ is an element of $\TT$ 
by \cref{Lem:ProjRes-in-T_Simple}.
\end{proof}

We are now ready to prove the main result of this section.
\begin{Theorem} \label{Thm:Triangulated-Hull-of-L_Closed-under-nu}
Let $A$ be a finite dimensional $k$-algebra.

The category $\Hstp A$ is the smallest triangulated subcategory of 
$\KKp{A}$ that contains $\LL_A$ and is closed under $\nu_{\mathcal{K}}$ and under isomorphisms.
\end{Theorem}
\begin{proof}
Note that $\HP A \subseteq \Hstp A \subseteq \Hp A$ and $\stp A \subseteq \PP_A$. 
By \cref{Lem:L-is-in-K_perp}.(3), we have that $\LL_A\subseteq \HP A \subseteq \Hstp A$. 
Furthermore, $\Hstp A$ is a triangulated subcategory of $\KKp A$ that is closed under isomorphisms
by \cref{Rem:Properties_H-HP-Hstp}.

We show that $\Hstp A$ is closed under $\nu_{\mathcal{K}}$. Let $F^\bt \in \Hstp A \cap \KKpbsp A$. 
Since we have $\nu_{\mathcal{K}} \left( \Kstp A \right) \simeq \Kstp A $, this implies 
$\nu_{\mathcal{K}} F^\bt \in \Kstperp A$.
Since $\KKmbp A \cap \Kstperp A$ is contained in $\Hstp A$, we obtain $\nu_{\mathcal{K}}(F^\bt) \in \Hstp A$.
In conclusion, $\Hstp A$ is a triangulated category that contains $\LL_A$ and is closed under $\nu_{\mathcal{K}}$
and under isomorphisms. 

Suppose that $\TT$ is another triangulated subcategory of $\KKp{A}$ that contains $\LL_A$ and is closed 
under $\nu_{\mathcal{K}}$ and under isomorphisms.
We show that $\Hstp A \subseteq \TT$.

By \cref{Lem:H-iff-Complex_in-perp} and \cref{Lem:Reduction-To-ProjRes} it suffices to show that the minimal projective 
resolution of every $A$-module $X \in \stperp A$ is an element of $\TT$.
Moreover, by \cref{Lem:ProjRes-if-CompFactors-in-T} it suffices to show that the minimal projective resolution of every 
composition factor of $X$ is an element of $\TT$. Let $S$ be such a composition factor.
Then $S$ is an element of $\stperp A$ by \cref{Lem:X_iff_CompFactor-in-P-perp}.
Using \cref{Lem:ProjRes-in-T_Simple_nu}, we now obtain that the minimal projective resolution of $S$ is an element of $\TT$.
\end{proof}

Similarly as for $\HP A$, we can characterize $\Hstp A$ using right perpendicular categories;
cf.\ \cref{Rem:HP-as-RightPerp}.
\begin{Remark} \label{Rem:Hstp-as-RightPerp}
The following are equivalent for a complex $F^\bt \in \KKp A$.
\begin{itemize}
\item[(1)] $F^\bt \in \Kstperp A$.
\item[(2)] $H^k(F^\bt) \in \stperp A$ for all $k \in \ZZ$.
\item[(3)] $H^k(F^\bt) \in (\stp A)^\perp$ for all $k \in \ZZ$.
\item[(4)] $F^\bt \in \mathcal{K}^b(\stp A)^\perp$.
\end{itemize}
In particular, we have $F^\bt \in \Hstp A$ if and only if $F^\bt \in \mathcal{K}^b(\stp A)^\perp$ and $F^\bt \in \Hp A$.
\end{Remark}

We also note the following for the Auslander-Reiten quiver of $\DDb A$.
\begin{Remark}   \label{Rem:Hstp_and_AR-components}  
Suppose that $A$ has finite global dimension.
An Auslander-Reiten triangle in $\KKbp A$ is of the following form;
cf.\ \cite[Theorem 1.4]{Happel_Triangles}.
\[
\nu(F^\bt)[-1] \to G^\bt \to F^\bt \to 
\]
Thus, $\Hstp A$ is the union of some Auslander-Reiten components in $\Hp A \simeq \KKbp A$.
\end{Remark}

We observe the following for the case of $\HP A = \Hstp A$.
\begin{Lemma} \label{Lem:Charact_nu-domdim>0}
We have $\HP A = \Hstp A$ if and only if $\PP_A = \stp A$.
\end{Lemma}
\begin{proof}
Suppose that $\PP_A = \stp A$. Then we have $\KPperp A = \Kstperp A$ so that $\HP A = \Hstp A$.

On the other hand, suppose that $\HP A = \Hstp A$. In particular, this means that $\KPperp A = \Kstperp A$.
Assume that $Z \in \PP_A$ is indecomposable and not an element of $\stp A$. 
Then we have $\soc(Z) \not\simeq \soc(Z')$ for all $Z' \in \stp A$ indecomposable.
We obtain $\Hom_A(\soc(Z),Z') = 0$ for all $Z' \in \stp A$ indecomposable. Hence, $\soc(Z) \in \stperp A$.
Using \cref{Lem:H-iff-Complex_in-perp}, the assumption $\KPperp A = \Kstperp A$ implies that $\soc(Z) \in \Pperp_A$.
However, we have $\Hom_A(\soc(Z),Z) \neq 0$, a contradiction. 
\end{proof}

Let $e$ be a basic idempotent element in $A$ such that $\mathrm{add}(eA) = \stp A$.
The algebra $eAe$ is called an \textit{associated self-injective algebra}; 
cf.\ \cite[Section 4]{DugasMartinezVilla_MoritaType}.
We give a characterization of $\Hstp A$ inside $\Hp A$ using the algebra $eAe$.
\begin{Lemma} \label{Lem:Charact_Hstp-In-Hp}
Let $e$ be a basic idempotent element in $A$ such that $\mathrm{add}(eA) = \stp A$ and suppose that $F^\bt \in \KK{A}$. 
Then $F^\bt \in \Kstperp A$ if and only if $(Fe)^\bt \in \KK{eAe}$ is acyclic.

Suppose that $F^\bt \in \Hp A$. Then $F^\bt \in \Hstp A$ if and only if $(Fe)^\bt \in \KK{eAe}$ is acyclic.
\end{Lemma}
\begin{proof} 
We use that $(Fe)^\bt \in \KK{eAe}$ is acyclic if and only if $\HH^k\big((Fe)^\bt\big) = 0$ for all $k \in \ZZ$.

Since $Ae$ is a projective left $A$-module, the functor $- \otimes_A Ae$ is exact. 
Thus, we have $\HH^k\big((Fe)^\bt\big) \simeq \HH^k(F^\bt)e$.
By assumption, $\nu(eA) \simeq eA$ so that
\[
\HH^k(F^\bt)e \simeq \Hom_A\big(eA,\, \HH^k(F^\bt)\big) \simeq \Du\,\Hom_A\big(\HH^k(F^\bt), \nu(eA)\big) \simeq \Du\, \Hom_A\big(\HH^k(F^\bt),\, eA).
\]
Note that $\Hom_A\big(\HH^k(F^\bt),\, eA) = 0$ if and only if $\HH^k(F^\bt) \in \stperp A = {}^\perp\big(\mathrm{add}(eA)\big)$.
Furthermore, we have $\HH^k(F^\bt) \in \stperp A$ for all $k \in \ZZ$ if and only if we have
$F^\bt \in \Kstperp A$ by \cref{Lem:H-iff-Complex_in-perp}.
\end{proof}

In \cite[Theorem 4.3]{FangHuKoenig_DerivedEquiv}, Fang, Hu and Koenig show that a derived equivalence
$\DDb A \to \DDb B$ restricts to an equivalence $\Kstp A \to \Kstp B$,
provided the two given algebras have $\nu$-dominant dimension at least one.
In particular, the associated self-injective algebras are derived equivalent in this setting.

Recall that $\Hp A \simeq \KKbp A \simeq \DDb A$, if $\gdim A < \infty$.
We have the following corollary of \cite[Theorem 4.3]{FangHuKoenig_DerivedEquiv} for our situation.
\begin{Corollary} \label{Cor:Kstp-characteristic_finite-gdim}
Let $A$ and $B$ be derived equivalent $k$-algebras, both of $\nu$-dominant dimension at least $1$.
Assume that $A$ and $B$ have finite global dimension.

Then any derived equivalence $\KKbp A \to \KKbp B$ restricts to an
equivalence of triangulated subcategories $\Hstp A \to \Hstp B$.
 
In particular, $\Hstp A = \HP A$ is a characteristic subcategory of $\KKbp A$.
\end{Corollary}

We aim to state a similar result for equivalences $\Hp A \to \Hp B$ 
without a restriction on the global dimension of $A$.
We follow the strategy used in \cite{FangHuKoenig_DerivedEquiv}.
As the main tool for proving the above theorem, they introduce the following subcategory.
\[
\mathscr{X}_A := \left\lbrace P^\bt \in \Kstp A | P^\bt \simeq \nu_A(P^\bt) \text{ in } \DDb A \right\rbrace
\]
Suppose that $A$ has $\nu$-dominant dimension at least $1$.
By \cite[Proposition 4.2]{FangHuKoenig_DerivedEquiv}, $\Kstp A$ is the smallest triangulated 
full subcategory of $\KKbp A$ that contains $\mathscr{X}_A$ and is closed under taking direct summands.

Furthermore, we need a way to restrict an equivalence $\Hp A \to \Hp B$ to bounded complexes. 
For this, we adapt the characterization of $\KKbp A$ inside $\KKmbp A$.
A complex $X^\bt \in \KKmbp A$ is an element of $\KKbp A$ if and only if for all $Y^\bt \in \KKmbp A$
there exists an $N \in \ZZ$ such that $\Hom_{\KKp A}(X^\bt,Y^\bt[-n]) = 0$ if $n < N$.
\begin{Lemma} \label{Lem:Kbproj_in_H}
Let $X^\bt \in \Hp A$.
\begin{itemize}
\item[(1)] We have $X^\bt \in \KKpbsp A$ if for all $Y^\bt \in \Hp A$ there exists an $N \in \ZZ$ such that
           $\Hom_{\Hp A}(X^\bt, Y^\bt[-n]) = 0$ if $n < N$.

\item[(2)] We have $X^\bt \in \KKmbp A$ if for all $Y^\bt \in \Hp A$ there exists an $N \in \ZZ$ such that 
           $\Hom_{\Hp A}(Y^\bt[n], X^\bt) = 0$ if $n < N$.

\item[(3)] Suppose that $X \in \Kstp A$. Let $Y^\bt \in \Hp A$. 
          Then there exists an $N \in \ZZ$ such that $\Hom_{\Hp A}(Y^\bt[n], X^\bt) = 0$ 
          and $\Hom_{\Hp A}(X^\bt, Y^\bt[-n]) = 0$ if $n < N$.
\end{itemize}
\end{Lemma}
\begin{proof}
\textit{Ad (1).} It suffices to show, that $X^\bt$ is bounded on the left.
    
    We assume that $X^\bt$ is unbounded on the left. Since $X^\bt \in\Hp A$, there exists an $N_0 \in \ZZ$ such that
	$\HH^n(X^\bt) = 0$ for $n < N_0$.
	Furthermore, there exists an $N < N_0$ such that $\ker(d_X^n)$ is not projective for $n < N$. 
	Otherwise the complex $X^\bt$ would be isomorphic to a complex which is left bounded by removing split direct summands.
	
    Let $\{S_1,\, \dots,\, S_l \}$ be a complete set of pairwise non-isomorphic simple $A$-modules. 
    For their direct sum, we write $S:= \bigoplus_{i=1}^l S_i$. Let $Y^\bt$ be the minimal projective resolution of $S$. 
    Note that $Y^\bt$ is an element of $\Hp A$.
	We show that $\Hom_{\Hp A}(X^\bt, Y^\bt[-n]) \neq 0$ for $n < N$. 

	Let $n < N$. We write $X:= \HH^n(\tau_{\leqslant n} X^\bt) = \Ker(d_X^{n+1}) \in \rmod A$ using that $n+1 \leq N < N_0$. 
	Suppose that $f$ is the composite of the natural projection $p: X \to X/\rad(X)$ and the natural embedding 
	$X/\rad(X) \to S$.
	Since $\HH^k(X^\bt) = 0$ and $\HH^k(Y^\bt[-n]) = 0$ for $k < n$, the morphism $f$ 
	lifts to a morphism $f^\bt$ of complexes.
	\[
	\begin{tikzcd}[row sep = .6cm]
	X^\bt \ar[ddd, "f^\bt"] & \cdots \ar[r] & X^{n-1} \ar[r] \ar[ddd, "f^{-1}"] & X^{n} \ar[rr] \ar[ddd, "f^0"] \ar[dr, twoheadrightarrow] & & X^{n+1} \ar[r] \ar[ddd] & \cdots \\
	& & & & X \ar[d, "f"] \ar[ur, hookrightarrow] \\
	& & & & S \\
	Y^\bt[-n]& \cdots \ar[r] & Y^{-1} \ar[r] & Y^0 \ar[rr] \ar[ur, twoheadrightarrow] & & 0 \ar[r] & \cdots 
	\end{tikzcd}		
	\]
	Assume that $f^\bt = 0$. This implies that $f$ factors through 
	the projective cover $P$ of $X/\rad(X) = \im(f)$. Since $p$ is surjective, we obtain a morphism $g : P \to X$.
	\[
     \begin{tikzcd}
                  & P \ar[d] \ar[ld, dashed, "g"'] \\
     X \ar[r,"p"] & X/\rad(X)
     \end{tikzcd}     	
	\]
	Using that $X/\rad(X) \simeq P/\rad(P)$, we obtain that $g$ is surjective so that $X$ is a direct summand of $P$.
	This is a contradiction to the choice of $\ker(d_X^{n+1}) = X$ as non-projective.
	Therefore, the morphism $f^\bt : X^\bt \rightarrow Y^\bt[-n]$ is non-zero.
	
\textit{Ad (2).} 
	We have $X^\bt \in \KKmbp A$ if and only if $X^\ast_\bt \in \KKpbspleft A$.
	
	We rename $X^\ast_\bt$ as $U^\bt$ via $U^k := X^\ast_{-k} \in \KKpbspleft A$
	so that $U^\bt[1]$ shifts the complex to the left. In contrast, $X^\ast_\bt[1]$ is a shift to the right.
    \[
    \begin{tikzcd}[column sep = .5cm, row sep = .4cm]
    U^\bt = X^\ast_\bt :         &               & \cdots \ar[r]            & U^{-1} = X^\ast_1 \ar[r]      & U^0 = X^\ast_0 \ar[r]           & U^1 = X^\ast_{-1} \ar[r] & \cdots \\
    U^\bt[1] = X^\ast_\bt[-1] :  & \cdots \ar[r] & U^{-1} = X^\ast_1 \ar[r] & U^0 = X^\ast_0 \ar[r]         & U^1 = X^\ast_{-1} \ar[r]        & \cdots                   &         
    \end{tikzcd}
    \]
	Since part (1) also holds for left
	$A$-modules, it suffices to show that for all $V^\bt \in \Hpleft A$ 
	there exists an $N \in \ZZ$ such that we have
	$\Hom_{\Hpleft A}(U^\bt, V^\bt[-n]) = 0$ if $n < N$.
	
    Let $V^\bt \in \Hpleft A$ and write $V^k = Y^\ast_{-k} \in \Hpleft A$ for a complex $Y^\bt \in \Hp A$. 
    By assumption, we have an $N \in \ZZ$ such that the following holds for $n<N$.
	\[
	0 =      \Hom_{\Hpleft A}(Y^\bt[n], X^\bt) 
	  \simeq \Hom_{\Hpleft A}(X^\ast_\bt, Y^\ast_\bt[n]) 
	  =      \Hom_{\Hpleft A}(U^\bt, V^\bt[-n]) 
	\]
\textit{Ad (3).} Suppose given $X^\bt \in \Kstp A$ and $Y^\bt \in \Hp A$. 
    Without loss of generality we may assume that $X^k = 0$ for $k > 0$.
    Let $l \in \ZZ_{\geqslant 0}$ such that $X^k = 0$ for $k < -l$.

    Let $N_1 \in \ZZ$ such that $\HH^n(Y^\bt) = 0$ for all $n < N_1$.
    Using that $X^{> 0} = 0$, \cref{Lem:L-is-in-K_perp}.(1) implies that
    $\Hom_{\Hp A}(Y^\bt[n],X^\bt) = 0$ if $n < N_1$.
    \[
     \begin{tikzcd}[column sep = .5cm]
     Y^\bt[n] \ar[d] &  \cdots \ar[r] & Y^{n-l-1} \ar[r] \ar[d] & Y^{n-l} \ar[r] \ar[d] & Y^{n-l+1} \ar[r] \ar[d] & \cdots \ar[r] & Y^{n-1} \ar[r] \ar[d] & Y^n \ar[r] \ar[d] & Y^{n+1} \ar[r] \ar[d] & \cdots \\  
     X^\bt           &                & 0 \ar[r]                & X^{-l} \ar[r]          & X^{l-1} \ar[r]          & \cdots \ar[r] & X^{-1} \ar[r]         & X^0 \ar[r]        & 0
     \end{tikzcd}         
    \]
    Let $N_2 \in \ZZ$ such that $\HH_{-n}(Y_\bt^\ast) = 0$ for $n < N_2$. 
    Using that $X^k = 0$ for $k < -l$, \cref{Lem:L-is-in-K_perp}.(1') implies that     
    $\Hom_{\Hp A}(X^\bt, Y^\bt[-n+l]) = 0$ if $n < N_2$.
    Note that $X_\bt^\ast \in \KKbileft A$ since $X^\bt \in \Kstp A$. 
    \[ \hspace*{-.3cm}
     \begin{tikzcd}[column sep = .38cm]
     X^\bt \ar[d] &                & 0 \ar[r] \ar[d]  & X^{-l} \ar[r] \ar[d] & X^{-l+1} \ar[r]\ar[d] & \cdots \ar[r] & X^{-1} \ar[r] \ar[d] & X^0 \ar[r] \ar[d] & 0 \ar[d]          &        \\
     Y^\bt[-n+l]  &  \cdots \ar[r] & Y^{-n-1} \ar[r]  & Y^{-n} \ar[r]        & Y^{-n+1} \ar[r]       & \cdots \ar[r] & Y^{-n+l-1} \ar[r]    & Y^{-n+l} \ar[r]   & Y^{-n+l+1} \ar[r] & \cdots  
     \end{tikzcd}         
    \]
    Let $N := \min\{N_1,N_2-l\}$. Then $\Hom_{\Hp A}(Y^\bt[n], X^\bt) = 0$ and 
    $\Hom_{\Hp A}(X^\bt, Y^\bt[-n]) = 0$ if $n < N$.
\end{proof}

In \cite[Proposition 5.2]{Rickard_DerEquivDerFunct} Rickard has shown that any standard derived equivalence
commutes with the Nakayama functor.
It seems unclear whether the same holds true for an equivalence between $\Hp A$ and $\Hp B$.
Therefore, we add a further assumption in contrast to \cite[Theorem 4.3]{FangHuKoenig_DerivedEquiv}.
In case that $\gdim A < \infty$, we are in the situation of \cite[Theorem 4.3]{FangHuKoenig_DerivedEquiv} where
the additional steps of the following theorem are not needed.
\begin{Theorem} \label{Thm:Kstp-characteristic_in_H}
Suppose given two finite dimensional $k$-algebras $A$ and $B$ both with $\nu$-dominant dimension at least $1$.
Let $\alpha : \Hp A \to \Hp B$ be a triangulated equivalence such that there is a natural isomorphism 
$\nu_B(\alpha(X^\bt)) \simeq \alpha(\nu_A X^\bt)$ for all $X \in \Kstp A$.

Then $\alpha$ restricts to an equivalence $\Kstp A \simeq \Kstp B$.
Moreover, $\alpha$ restricts to an equivalence between $\HP A = \Hstp A$ and $\HP B = \Hstp B$.
\end{Theorem}
\begin{proof}
Suppose given $X^\bt \in \Kstp A$. By \cref{Lem:Kbproj_in_H}.$(3)$, there is an $N \in \ZZ$ such that
we have $\Hom_{\Hp A}(Y^\bt[n], X^\bt) = 0$ and $\Hom_{\Hp A}(X^\bt, Y^\bt[-n]) = 0$ 
for all $n < N$ and $Y^\bt \in \Hp A$. 

Let $n < N$ and $Z^\bt \in \Hp B$. 
Since $\alpha$ is an equivalence, there exists a $Y^\bt \in \Hp A$ with $\alpha(Y^\bt) = Z^\bt$.
We have the following.
\begin{align*}
&\Hom_{\Hp B}(Z^\bt[n], \alpha(X^\bt)) \simeq \Hom_{\Hp A}(Y^\bt[n], X^\bt) = 0 \\
&\Hom_{\Hp B}(\alpha(X^\bt), Z^\bt[-n]) \simeq \Hom_{\Hp A}(X^\bt, Y^\bt[-n]) = 0.
\end{align*}
Hence, by \cref{Lem:Kbproj_in_H}.$(1,2)$, we obtain $\alpha(X^\bt) \in \KKbp B$.

Now, suppose that $X^\bt \in \mathscr{X}_A$.
By assumption, we have $\alpha(X^\bt) \simeq \alpha(\nu_A(X^\bt)) \simeq \nu_B(\alpha(X^\bt))$,
which implies that $\alpha(X^\bt) \in \mathscr{X}_B$.
In conclusion, we have shown, that $\alpha(\mathscr{X}_A) \subseteq \mathscr{X}_B$.
By \cite[Proposition 4.2]{FangHuKoenig_DerivedEquiv}, we therefore obtain $\alpha(\Kstp A) \subseteq \Kstp B$.

Let $\beta$ be a quasi-inverse of $\alpha$. Repeating the arguments above, we obtain
$\beta(\Kstp B) \subseteq \Kstp A$.
Together, we can conclude that $\alpha$ induces an equivalence $\Kstp A \simeq \Kstp B$.

Recall that $\Hstp A$ is the full subcategory of $\Hp A$ with objects in $\Kstperp A$. 
Hence, $\alpha$ also induces an equivalence $\Hstp A \simeq \Hstp B$.
Since $\nddim A \geq 1$ and $\nddim B \geq 1$, we have $\Hstp A = \HP A$ and $\Hstp B = \HP B$;
cf.\ \cref{Lem:Charact_nu-domdim>0}. 
\end{proof}

\section{Self-injective algebras}
\label{Sec:SelfInjective}

In this short section, we discuss the case of self-injective algebras.
Recall that the category $\stmod A$ is triangulated, if $A$ is self-injective.

\begin{Lemma} \label{Lem:F-TriangulatedFunctor}
Let $A$ be self-injective.

Then $\LL_A$ is a triangulated subcategory of $\KKp A$ and
$\FF : \stmod A \to \LL_A$ is an equivalence of triangulated categories.
\end{Lemma}
\begin{proof}
If $A$ is self-injective, $\Hom_A(-,A)$ is exact so that we have
\begin{align*}
\HH^{<0}(F^\bt) &= 0 \Leftrightarrow \mathrm{H}^{<0}(F_\bt^\ast) = 0 \\
\HH^{\geqslant 0}(F_\bt^\ast) &= 0 \Leftrightarrow \mathrm{H}^{\geqslant 0}(F^\bt) = 0.
\end{align*}
Together, we obtain that 
$\LL_A = \{ F^\bt \in \mathrm{K(proj}\,A) \;\vert\; \mathrm{H}^{k}(F^\bt) = 0, \, k\in \ZZ \}$.
Hence, $\LL_A$ is closed under shifts.
As another consequence, a morphism $f^\bt : F^\bt \to G^\bt$ in $\LL_A$ 
is a quasi-isomorphism and therefore $C(f)^\bt \in \LL_A$.
In conclusion, $\LL_A$ is a triangulated subcategory of $\KKp A$.
It remains to show that $\FF$ is triangulated.

Suppose given $X \in \stmod A$ with $F_{\Omega^{-1}(X)}^\bt \in \LL_A$. Then
\[
\cdots \to F_{\Omega^{-1}(X)}^{-2} \to F_{\Omega^{-1}(X)}^{-1} \to F_{\Omega^{-1}(X)}^0 \to 0
\]
is a projective resolution of $\Omega^{-1}(X)$. 
Therefore, we have that 
$\HH^{-1}\left(\tau_{\leqslant -1}\, F_{\Omega^{-1}(X)}^\bt\right) \stackrel{\mathrm{st}}{\simeq} X$ so that 
\hbox{$F_X^\bt \cong F_{\Omega^{-1}(X)}^{\bt}[-1]$} or equivalently 
$F_X^{\bt}[1] \cong F_{\Omega^{-1}(X)}^\bt$. Hence, $\FF$ commutes with the shift.

If $A$ is self-injective, every  short exact sequence is perfect exact. 
Moreover, every distinguished triangle in $\stmod A$ is induced by a short exact sequence.
Therefore, \cref{Prop:PerfSeq-DistTriang} shows that $\FF$ maps distinguished triangles in $\stmod A$ 
to distinguished triangles in $\LL_A$.
\end{proof}

Recall that $\KKptac A$ denotes the category of totally acyclic complexes in $\KKp A$.
That is, $F^\bt \in \KKptac A$ if $\HH^\bt (F^\bt) = 0$ and $\HH_\bt(F_\bt^\ast) = 0$.

\begin{Lemma} \label{Lem:KKptac_LargestTriangSubcat}
Suppose given $F^\bt \in \LL_A$. We have $F^\bt \in \KKptac A$ if $F^\bt[k] \in \LL_A$ for all $k \in \ZZ$.
In particular, the category $\KKptac A$ is the largest subcategory of $\LL_A$ that is triangulated as a subcategory of $\KKp A$.
\end{Lemma}
\begin{proof}
Suppose given $F^\bt \in \LL_A$ with $F^\bt[k] \in \LL_A$ for all $k \in \ZZ$.
We have the following for all $k \in \ZZ$.
\begin{align*}
\HH^k(F^\bt) = \HH^{-1}(F^\bt[k+1]) &= 0 \\
\HH_k((F^\ast)_\bt) = \HH_0((F^\ast)_\bt[k]) = \HH_0\left((F^\bt[k])^\ast\right) &= 0
\end{align*}
In conclusion, $F^\bt \in \KKptac A$.

Let $\TT$ be a subcategory of $\LL_A$ that is triangulated. In particular, $\TT$ is closed under shifts. 
Now, the above shows that $\TT$ is contained in $\KKptac A$.
\end{proof}

The connection between the different categories discussed so far can be visualized as follows.
\[
\begin{tikzcd}[column sep = .5cm]
\KKptac A \ar[r, hookrightarrow]  & \LL_A \ar[r, hookrightarrow] & \HP A \ar[r, hookrightarrow] & \Hstp A \ar[r, hookrightarrow] & \Hp A \ar[r, hookrightarrow] & \KKp A 
\end{tikzcd}
\]
With the exception of $\LL_A$, all of these categories are triangulated for all finite dimensional algebras.
Moreover, in general, all inclusions are proper.
However, some of these categories coincide if and only if $A$ is self-injective.

\begin{Theorem} \label{Thm:Charact_SelfInjective} 
The following are equivalent for a finite dimensional algebra $A$.
\begin{itemize}
\item[(1)] $A$ is self-injective.

\item[(2)] $\LL_A$ is a triangulated subcategory of $\KKp A$.

\item[(3)] $\LL_A = \HP A$.

\item[(4)] $\LL_A$ is closed under taking shifts in $\KKp A$.

\item[(5)] $\LL_A = \KKptac A$.
\end{itemize}
If one of the above conditions holds, $\mathcal{F} : \stmod A \to \LL_A$ is an equivalence of triangulated categories.
Furthermore, we have $\KKptac A = \LL_A = \HP A = \Hstp A$.
\end{Theorem}
\begin{proof} 
It was shown in \cref{Lem:F-TriangulatedFunctor} that condition (2) holds if $A$ is self-injective.
The implication from (2) to (3) holds by \cref{Thm:Triangulated-Hull-of-L}.
Since $\HP A$ is a triangulated subcategory of $\KKp A$, condition (3) implies condition (4).
By \cref{Lem:KKptac_LargestTriangSubcat}, the implication (4) $\Rightarrow$ (5) holds as well.

\textit{We verify the implication $(5) \Rightarrow (1)$}.
An algebra $A$ is self-injective if and only if every finitely generated module is reflexive; 
cf.\cite[IV. Proposition 3.4]{AuslanderReitenSmalo}.
Let $X \in \rmod A$. 
We show that $X$ is reflexive, that is $(X^\ast)^\ast \simeq X$. 
We have 
\[
X^\ast = \left(\HH^0(\tau_{\leqslant 0} F_X^\bt)\right)^\ast \simeq \HH_0(\tleq F_X^{\bt,\ast}) \simeq \HH_1(\tau_{\geqslant 1} F_X^{\bt,\ast})
\]
since $(-)^\ast = \Hom_A(-,A)$ is left exact and $F_X^\bt \in \LL_A$.
\[
\begin{tikzcd}[column sep = .2cm, row sep = .4cm]
\cdots \ar[rr] & & F_2^\ast \ar[rr] & & F_1^\ast \ar[rr] \ar[dr, twoheadrightarrow] &                                & F_0^\ast \ar[rr] & & F_{-1}^\ast \ar[rr] & &  \cdots \\
               & &                  & &                                             & X^\ast \ar[ur, hookrightarrow] &                  & &                     & &
\end{tikzcd}
\]
Using that $\HH^0(F_X^\bt) = 0$ since $F_X^\bt \in \LL_A = \KKptac A$, we similarly obtain 
\[
(X^\ast)^\ast = \left(\HH_1(\tau_{\geqslant 1} F_X^{\bt,\ast})\right)^\ast \simeq \HH^1(\tau_{\geqslant 1} F_X^\bt) \simeq \HH^0(\tau_{\leqslant 0} F_X^\bt) = X.
\]
In conclusion, all conditions are equivalent. It was shown in \cref{Lem:F-TriangulatedFunctor} that 
$\FF : \stmod A \to \LL_A$ is an equivalence of triangulated categories if $A$ is self-injective.
Furthermore, in this case we have $\stp A = \PP_A = \proj A$ so that $\Hstp A = \HP A$.
\end{proof}

\section{Stable equivalences of Morita type}
\label{Sec:StEquivMoritaType}

In this section we aim to show that a stable equivalence of Morita type $\stmod A \to \stmod B$ induces equivalences 
between $\HP A \to \HP B$ and $\Hstp A \to \Hstp B$.

We recall the definition of stable equivalences of Morita type and begin by collecting some technical properties.
\begin{Definition}[Brou\'e] \label{Def:StEquivMoritaType}
Let ${}_A M_B$ and ${}_B N_A$ be bimodules such that ${}_A M$, $M_B$, ${}_B N$ and $N_A$ are projective.
We say that $M$ and $N$ induce a \textit{stable equivalence of Morita type} if 
\[ 
{}_A M \otimes_B N_A \simeq A \oplus P \text{ and } {}_B N \otimes_A M_B \simeq A \oplus Q
\]
as bimodules such that ${}_A P_A$ and ${}_B Q_B$ are projective bimodules.
\end{Definition}

\begin{Lemma} \label{Lem:NatIsom_BimodulesMoritaType}
Suppose ${}_A M_B$ and ${}_B N_A$ are bimodules that induce a stable equivalence of Morita type 
such that $M$ and $N$ do not have any non-zero projective bimodule as direct summand.

The following holds for $X \in \rmod A$.
\begin{itemize}
\item[(1)] $X \otimes_A M$ is projective-injective as a $B$-module if $X \in \PP_A$.

\item[(2)] $X \otimes_A M$ is strongly projective-injective as a $B$-module if $X \in \stp A$.
\end{itemize}
\end{Lemma}
\begin{proof}
\textit{Ad (1).} 
We use that $- \otimes_A M$ is right adjoint to $- \otimes_B N$; cf.\ \cite[Lemma 4.1]{ChenPanXi_MoritaType}.

If $X \in \rmod A$ is injective, $\Hom_B(-,X\otimes_A M) \simeq \Hom_A(- \otimes_B N, X)$ is an exact functor
since $- \otimes_B N$ is exact. Thus, $X\otimes_A M \in \inj B$.

If $X \in \proj A$, then $X \otimes_A M$ is a projective $B$-module since $M_B \in \proj B$. 
Therefore, we have $X \otimes_A M \in \PP_B$ if $X \in \PP_A$.

\textit{Ad (2).} Let $X \in \stp A$.
Using \cite[Lemma 4.1]{ChenPanXi_MoritaType} again, we have a natural isomorphism 
$\nu_B^k(X \otimes_A M_B) \simeq \nu_A^k(X) \otimes_A M_B$ as shown in \cite[Lemma 3.3]{DugasMartinezVilla_MoritaType}
so that $\nu_B^k(X \otimes_A M_B)$ is a projective $B$-module
for $k \in \ZZ$. Thus, $X \otimes_A M_B \in \stp B$.
\end{proof}

For part (3) of the following lemma, recall that $\stp A \subseteq \PP_A$.
\begin{Lemma} \label{Lem:MoritaType-IsoOnPerpP} 
Suppose ${}_A M_B$ and ${}_B N_A$ are bimodules that induce a stable equivalence of Morita type 
such that $M$ and $N$ do not have any non-zero projective bimodule as direct summand.
The following holds.
\begin{itemize}
\item[(1)] $X \otimes_A M \in \Pperp_B$ if $X \in \Pperp_A$.

\item[(2)] $X \otimes_A M \in \stperp A$ if $X \in \stperp A$.

\item[(3)] $X \otimes_A M \otimes_B N \simeq X$ if $X \in \stperp A$.
\end{itemize}
\end{Lemma}
\begin{proof}
\textit{Ad (1) and (2).} 
We use that $- \otimes_A M$ is left adjoint to $- \otimes_B N$ by \cite[Lemma 4.1]{ChenPanXi_MoritaType}.
For $Z \in \PP_B$ or $Z \in \stp B$ we have
\[
\Hom_B(X \otimes_A M, Z) \simeq \Hom_A(X, Z \otimes_B N ) = 0
\] 
since $Z\otimes_B N \in \PP_A$ or $Z \otimes_B N \in \stp A$ respectively by \cref{Lem:NatIsom_BimodulesMoritaType}.

\textit{Ad (3).} We have
\[
X \otimes_A M \otimes_B N \simeq X \otimes_A (A_A \oplus P_A) \simeq X \oplus (X \otimes_A P_A).
\]
By \cite[Lemma 3.1.(2,3)]{HuXi_DerivedMoritaTypeII} we know that $X \otimes_A P_A \in \stp A$.
On the other hand, by part (2) we have $X \otimes_A M \in \stperp B$.
Similarly, we also have $Y \otimes_B N \in \stperp A$ for all $Y \in \stperp B$.
Together we obtain $X \otimes_A M \otimes_B N \in \stperp A$.
Thus, $X \otimes_A P_A$ must be zero and $X \otimes_A M \otimes_B N \simeq X$.
\end{proof}

\begin{Lemma} \label{Lem:Homology-ExactFunctor}
Suppose ${}_A M_B$ and ${}_B N_A$ are bimodules that induce a stable equivalence of Morita type 
such that $M$ and $N$ do not have any non-zero projective bimodule as direct summand.

The following holds for a complex $F^\bt \in \KKp A$ and all $k \in \ZZ$.
\begin{itemize}
\item[(1)] $\HH^k(F^\bt \otimes_A M) = 0$ if $\HH^k(F^\bt) = 0$.
\item[(2)] $\HH_k\big((F^\bt \otimes_A M)^\ast\big) = 0$ if $\HH_k(F_\bt^\ast) = 0$.
\end{itemize}
\end{Lemma}
\begin{proof}
Using that $-\otimes_A M$ is an exact functor, we have $\HH^k(F^\bt \otimes_A M) \simeq \HH^k(F^\bt) \otimes_A M$.

Because of \cite[Lemma 4.1]{ChenPanXi_MoritaType}, we have a natural isomorphism 
$(F^\bt \otimes_A M)^\ast \simeq N \otimes_A F_\bt^\ast$ as shown in \cite[Lemma 3.3]{DugasMartinezVilla_MoritaType}.
Using that $N \otimes_A -$ is an exact functor, we obtain
\belowdisplayskip=-13pt
\[
\HH_k\big((F^\bt \otimes_A M)^\ast\big) \simeq \HH_k(N \otimes_A F_\bt^\ast) \simeq N \otimes_A \HH_k(F_\bt^\ast).
\]
\end{proof}

In general, a stable equivalence of Morita type does not induce an equivalence between $\Hp A$ and $\Hp B$.
Similarly, it does not induce an equivalence between $\KKp A$ and $\KKp B$.
However, we have the following main result of this section.
\begin{Theorem} \label{Thm:InducedEquivalenceOnLAndH}
Suppose ${}_A M_B$ and ${}_B N_A$ are bimodules that induce a stable equivalence of Morita type 
such that $M$ and $N$ do not have any non-zero projective bimodule as direct summand.
\begin{itemize}
\item[(1)] Applying $-\otimes_A M$ componentwise induces an equivalence of categories $\LL_A \rightarrow \LL_B$.
           If $A$ and $B$ are self-injective, this is an equivalence of triangulated categories.
           
\item[(2)] Applying $-\otimes_A M$ componentwise induces an equivalence of triangulated categories 
           \[ \HP A \rightarrow \HP B. \]

\item[(3)] Applying $-\otimes_A M$ componentwise induces an equivalence of triangulated categories 
           \[ \Hstp A \rightarrow \Hstp B. \]
\end{itemize}
\end{Theorem}
\begin{proof}
\textit{Ad (1).} 
Note that $- \otimes_A M$ induces a functor
$\KKp A \to \KKp B$ by componentwise application.
Now, \cref{Lem:Homology-ExactFunctor} shows that $-\otimes_A M$ induces a well-defined functor $\LL_A \rightarrow \LL_B$.
Consider the following diagram.
\[
\begin{tikzcd}[row sep=1cm, column sep=2cm]
\underline{\mathrm{mod}}\,A \ar[d, "\mathcal{F}"', "\sim" sloped] \ar[r, "\sim"', "- \otimes_A M"] & \underline{\mathrm{mod}}\,B \ar[d, "\sim"' sloped, "\mathcal{F}"] \\
\LL_A \ar[r, "- \otimes_A M"]& \LL_B
\end{tikzcd}
\]
Recall that the quasi-inverse of $\FF$ is given by $\HH^0(\tleq(-))$; cf.\ \cref{Thm:Equivalence_F}.
Let $F^\bt \in \LL_A$. Since $-\otimes_A M$ is exact, we have 
$ \HH^0\left(\tau_{\leqslant 0} \left(F^\bt \otimes_A M\right)\right)) \cong \HH^0\left(\tau_{\leqslant 0} F^\bt\right) \otimes_A M $ 
so that the diagram commutes.
This shows that $-\otimes_A M$ induces an equivalence of categories $\LL_A \rightarrow \LL_B$.

Furthermore, if $A$ and $B$ are self-injective, the equivalence $\FF$ is triangulated by \cref{Lem:F-TriangulatedFunctor}.
Thus, this diagram shows that the functor $-\otimes_A M$ induces an equivalence of triangulated categories $\LL_A \rightarrow \LL_B$.

\textit{Ad (2) and (3).}
By \cref{Lem:Homology-ExactFunctor}, the functor $-\otimes_A M$ induces a functor $\Hp A \to \Hp B$.
Since $-\otimes_A M$ is applied componentwise, this is a triangulated functor.

Suppose given $F^\bt \in \HP A$. We verify that $F^\bt \otimes_A M \in \KPperp B$.
By \cref{Lem:H-iff-Complex_in-perp} it suffices to show that $\HH^k(F^\bt \otimes_A M) \simeq \HH^k(F^\bt) \otimes_A M$
is an element of $\Pperp_B$ for $k \in \ZZ$.
However, this holds by \cref{Lem:MoritaType-IsoOnPerpP}.(1) since 
$\HH^k(F^\bt) \in \Pperp_A$ by \cref{Lem:H-iff-Complex_in-perp}.
In conclusion, $- \otimes_A M$ induces a functor $\HP A \to \HP B$.

Similarly, $- \otimes_A M$ induces a functor $\Hstp A \to \Hstp B$.
It remains to show that these are equivalences of triangulated categories.
Recall that $\HP A$ is contained in $\Hstp A$.

Let $F^\bt \in \Hstp A$ with $r \in \ZZ$ such that $\HH_{\geqslant r}(F_\bt^\ast) = 0$. 
We verify that we have a natural isomorphism $F^\bt \otimes_A M \otimes_B N \simeq F^\bt$ by induction on 
${N := \left|\left\lbrace j \in \ZZ_{< r} \left|\, \HH^j(F^\bt) \neq 0 \right\rbrace \right.\right|}$.
Since $F^\bt \in \Hp A$, we know that $\HH^\bt(F^\bt)$ is left bounded and therefore $N < \infty$.
We write $G^\bt := F^\bt \otimes_A M \otimes_B N$.

Let $N = 0$. Then $F^\bt[r] \in \LL_A$ and the assertion holds by part (1) since $-\otimes_A M$ commutes with the shift.

Let $N > 0$ and $k \in \ZZ_{< r}$ minimal such that $\HH^k(F^\bt) \neq 0$.
By \cref{Lem:RemovingCohomologyOfF}, we have a distinguished triangle
\[
P^\bt[-k] \to F^\bt \to C^\bt \to
\]
with $P^\bt$ a projective resolution of $\HH^k(F^\bt)$.
Moreover, $\HH^j(C^\bt) = 0$ for $j \leq k$ and $\tau_{\geqslant k} C^\bt = \tau_{\geqslant k} F^\bt$.

Applying $- \otimes_A M \otimes_B N$ to this triangle, we obtain a new distinguished triangle.
\[
P^\bt[-k] \otimes_A M \otimes_B N \to G^\bt \to C^\bt \otimes_A M \otimes_B N \to
\]
By \cref{Lem:MoritaType-IsoOnPerpP}.(3), we have $\HH^k(F^\bt) \otimes_A M \otimes_B N \simeq \HH^k(F^\bt)$ 
since $\HH^k(F^\bt) \in \stperp A$ by \cref{Lem:H-iff-Complex_in-perp}.
Hence, we have a natural isomorphism $P^\bt \otimes_A M \otimes_B N \simeq P^\bt$ in $\KKp A$ 
and obtain the following distinguished triangle.
\[
P^\bt[-k] \to G^\bt \to C^\bt \otimes_A M \otimes_B N \to
\]
By induction, we can assume that there is a natural isomorphism $ C^\bt \otimes_A M \otimes_B N \simeq C^\bt$.
This induces another distinguished triangle.
\[
\begin{tikzcd}
P^\bt[-k] \ar[r] \ar[d, equal] & G^\bt \ar[r] \ar[d, dashed] & C^\bt \ar[r] \ar[d, equal] & \; \\
P^\bt[-k] \ar[r]               & F^\bt \ar[r]                & C^\bt \ar[r]               & \;
\end{tikzcd}
\]
The induced morphism now yields a natural isomorphism $G^\bt = F^\bt \otimes_A M \otimes_B N \simeq F^\bt$.
\end{proof}

\textbf{Acknowledgements.} The results of this article are part of the author's PhD thesis \cite{Nitsche_PhD}.
The author would like to thank Steffen Koenig, Kiriko Kato and Yuming Liu for helpful comments and suggestions.

\end{document}